\documentclass[11pt]{amsart}
\usepackage{amssymb,amsmath,amscd,latexsym,epsfig,color,hyperref,graphics}
\usepackage[all,cmtip]{xy}

\newtheorem{theorem}{Theorem}[section]
\newtheorem{proposition}[theorem]{Proposition}
\newtheorem{corollary}[theorem]{Corollary}

\newtheorem{lemma}[theorem]{Lemma}

\newtheorem{problem}[theorem]{Problem}

\theoremstyle{definition}
\newtheorem{definition}[theorem]{Definition}
\newtheorem{remark}[theorem]{Remark}
\newtheorem{example}[theorem]{Example}

\def\NN{\ensuremath{\mathbb{N}}}
\def\ZZ{\ensuremath{\mathbb{Z}}}
 
\def\RR{\ensuremath{\mathbb{R}}}
\newcommand{\CC}{{\mathbb C}}

\newcommand{\V}{{\mathcal V}}
\newcommand{\B}{{\mathcal B}}

\def\I{\ensuremath{\mathcal{I}}}

\def\aa{\ensuremath{{\bf{a}}}}
\def\bb{\ensuremath{{\bf{b}}}}
\def\cc{\ensuremath{{\bf{c}}}}
\def\dd{\ensuremath{{\bf{d}}}}
\def\ee{\ensuremath{{\bf{e}}}}
\def\ff{\ensuremath{{\bf{f}}}}

\def\pp{\ensuremath{{\bf{p}}}}

\def\ss{\ensuremath{{\bf{s}}}}

\def\vv{\ensuremath{{\bf{v}}}}

\def\xx{\ensuremath{{\bf{x}}}}
\def\yy{\ensuremath{{\bf{y}}}}

\def\max{\ensuremath{\textup{max}}}

\def\conv{\ensuremath{\textup{conv}}}
\def\cl{\ensuremath{\textup{cl}}}
\def\deg{\ensuremath{\textup{deg}}}

\title{Theta Bodies for Polynomial Ideals} \author{Jo{\~a}o Gouveia}
\address{Department of Mathematics, University of Washington, Box
  354350, Seattle, WA 98195, USA, and CMUC, Department of Mathematics,
  University of Coimbra, 3001-454 Coimbra, Portugal}
\email{jgouveia@math.washington.edu} \author{Pablo A. Parrilo}
\address{Department of Electrical Engineering and Computer Science,
  Laboratory for Information and Decision Systems, Massachusetts
  Institute of Technology, 77 Massachusetts Avenue, Cambridge, MA
  02139-4307, USA} \email{parrilo@mit.edu} \author{Rekha R. Thomas}
\address{Department of Mathematics, University of Washington, Box
  354350, Seattle, WA 98195, USA} \email{thomas@math.washington.edu}
\thanks{All authors were partially supported by the NSF Focused
  Research Group grant (DMS-0757371, DMS-0757207). Gouveia was also
  supported by Funda{\c c}{\~ a}o para a Ci{\^ e}ncia e Tecnologia,
  and Thomas by the Robert R. and Elaine K. Phelps Endowed
  Professorship.}  \date{\today}

\begin{document}

\begin{abstract}
  Inspired by a question of Lov{\'a}sz, we introduce a hierarchy of
  nested semidefinite relaxations of the convex hull of real solutions
  to an arbitrary polynomial ideal, called theta bodies of the ideal.
  These relaxations generalize Lov{\'a}sz's construction of the theta
  body of a graph. 
  We establish a relationship between theta bodies and Lasserre's
  relaxations for real varieties which allows, in many cases, for
  theta bodies to be expressed as feasible regions of semidefinite
  programs.  Examples from combinatorial optimization are given.
  Lov{\'a}sz asked to characterize ideals for which the first theta
  body equals the closure of the convex hull of its real variety. We
  answer this question for vanishing ideals of finite point sets via
  several equivalent characterizations. We also give a geometric
  description of the first theta body for all ideals.
\end{abstract}
\maketitle

\section{Introduction}

A central concern in optimization is to understand $\conv(S)$, the
convex hull of the set of feasible solutions $S$, to a given problem.
In many instances, the set of feasible solutions to an optimization
problem is the set of real solutions to a polynomial system: $f_1(\xx)
= f_2(\xx) = \cdots = f_m(\xx) = 0$, where $f_1, \ldots, f_m \in
\RR[\xx] := \RR[x_1,\ldots,x_n]$. This set is the {\em real variety},
$\V_\RR(I)$, of the {\em ideal} $I$ in $\RR[\xx]$ generated by $f_1,
\ldots, f_m$, and it is often necessary to compute or represent
$\conv(\V_\RR(I))$ exactly or at least approximately. 

Recall that $\cl(\conv(\V_\RR(I)))$, the closure of
$\conv(\V_\RR(I))$, is cut out by the inequalities $f(\xx) \geq 0$ as
$f$ runs over all linear polynomials that are non-negative on
$\V_{\RR}(I)$.  (Call $f \in \RR[\xx]$ a {\em linear} polynomial if it
is affine linear of the form $f=a_0 + \sum_{i=1}^{n} a_i x_i$.) A
classical certificate for the non-negativity of a polynomial $f$ on
$\V_\RR(I)$ is the existence of a {\em sum of squares} ({\em sos})
polynomial $\sum_{j=1}^{t} h_j^2$ that is congruent to $f$ mod $I$
(i.e., $f-\sum_{j=1}^{t} h_j^2 \in I$), written as $f \equiv
\sum_{j=1}^{t} h_j^2$ mod $I$. If this is the case, we say that $f$
{\em is sos mod} $I$.  Hence a natural relaxation of
$\cl(\conv(\V_\RR(I)))$ is the closed convex set:
\begin{equation} \label{eqn:sos relaxation} \{ \xx \in \RR^n \,:\,
  f(\xx) \geq 0 \,\,\forall\,\,f \textup{ linear and sos mod } I \}.
\end{equation}
Depending on $I$, (\ref{eqn:sos relaxation}) may be strictly larger
than $\cl(\conv(\V_\RR(I)))$ since there may be polynomials that are
non-negative on $\V_\RR(I)$ but not sos mod $I$. However, in many
interesting cases, (\ref{eqn:sos relaxation}) will equal
$\cl(\conv(\V_\RR(I)))$. By bounding the degree of the $h_j$'s that
appear in the sos representations, and gradually increasing this
bound, we obtain a hierarchy of relaxations to
$\cl(\conv(\V_\RR(I)))$. In \cite{Lovasz}, Lov{\'a}sz asked a question
that leads to the study of this hierarchy. To explain it, we first
introduce some definitions.

\begin{definition} \label{def:perfect ideals} Let $f$ be a polynomial
  in $\RR[\xx]$, $I$ be an ideal in $\RR[\xx]$ with real variety
  $\V_{\RR}(I) := \{\ss \in \RR^n \,:\, f(\ss) = 0 \,\,\forall \,\,f
  \in I \}$, and let $\RR[\xx]_k$ denote the set of polynomials in
  $\RR[\xx]$ of degree at most $k$.
\begin{enumerate}
\item The polynomial $f$ is {\bf $k$-sos}
  mod $I$ if there exists $h_1,\ldots,h_t \in \RR[\xx]_k$ for some $t$
  such that $f \equiv \sum_{j=1}^t h_j^2 \,\,\textup{mod}\,\,I$.
  
\item The ideal $I$ is {\bf $k$-sos} if {\em every} polynomial that is
  non-negative on $\V_\RR(I)$ is $k$-sos mod $I$. If every polynomial
  of degree at most $d$ that is non-negative on $\V_{\RR}(I)$ is $k$-sos mod
  $I$, we say that $I$ is {\bf $(d,k)$-sos}.
\end{enumerate}
\end{definition}

\begin{example} \label{ex:runningex} Consider the principal ideal $I =
  \langle x_1^2x_2 - 1 \rangle \subset \RR[x_1,x_2]$. Then
  $\textup{conv}(\V_{\RR}(I)) = \{ (s_1,s_2) \in \RR^2 \,:\, s_2 > 0
  \}$, and any linear polynomial that is non-negative over
  $\V_{\RR}(I)$ is of the form $\alpha x_2 + \beta$, where $\alpha,
  \beta \geq 0$. Since $\alpha x_2 + \beta \equiv (\sqrt{\alpha}
  x_1x_2)^2 + (\sqrt{\beta})^2$ mod $I$, $I$ is $(1,2)$-sos. Check
  that $x_2$ is not $1$-sos mod $I$ and so, $I$ is not $(1,1)$-sos.
\end{example}

In \cite{Lovasz}, Lov{\'a}sz asked the following question.

\begin{problem} \cite[Problem 8.3]{Lovasz} \label{prob:lovasz} Which
  ideals in $\RR[\xx]$ are $(1,1)$-sos? How about
  $(1,k)$-sos?
\end{problem}

The geometry behind the above algebraic question leads to a natural
hierarchy of relaxations of $\textup{conv}(\V_{\RR}(I))$ which we now
introduce. The name comes from earlier work of Lov{\'a}sz and will be
explained in Section~\ref{sec:examples}.

\begin{definition} \label{def:theta} 
\begin{enumerate}
\item For a positive integer $k$, the $k$-th {\bf theta body} of an
  ideal $I \subseteq \RR[\xx]$ is
$$\textup{TH}_k(I) := \{ \xx \in \RR^n \,:\,
  f(\xx) \geq 0 \,\,\textup{for every linear}
  \,\,f\,\,\textup{that is $k$-sos mod} \,\,I \}.$$
\item An ideal $I \subseteq \RR[\xx]$ is {\bf $\textup{TH}_k$-exact}
  if $\textup{TH}_k(I)$ equals $\cl(\textup{conv}(\V_{\RR}(I)))$.
\item The {\bf theta-rank} of $I$ is the smallest $k$ for which $I$ is 
  $\textup{TH}_k$-exact.
\end{enumerate}
\end{definition}

By definition, $\textup{TH}_1(I) \supseteq \textup{TH}_2(I) \supseteq
\cdots \supseteq \textup{conv}(\V_{\RR}(I))$. As seen in
Example~\ref{ex:runningex}, $\textup{conv}(\V_{\RR}(I))$ may not be
closed while the theta bodies are. Therefore, the theta-body sequence
of $I$ can converge, if at all, only to
$\textup{cl}(\conv(\V_\RR(I)))$. 

A natural question at this point is whether
the algebraic notion of an ideal being $(1,k)$-sos is equivalent to
the geometric notion of being $\textup{TH}_k$-exact.

\begin{lemma} \label{lem:1ksos implies thetakexact} If an ideal $I
  \subseteq \RR[\xx]$ is $(1,k)$-sos then it is $\textup{TH}_k$-exact.
\end{lemma}

\begin{proof}
  Let $I$ be $(1,k)$-sos and $\ss \in \RR^n$ be not in
  $\textrm{cl}(\conv(\V_{\RR}(I)))$. By the {\em separation theorem}
  \cite[Theorem III.1.3]{Barvinok} there exists a linear polynomial
  $f$, non-negative over $\textrm{cl}(\conv(\V_{\RR}(I)))$, such that
  $f(\ss)<0$. However, since $I$ is $(1,k)$-sos, $f$ is $k$-sos mod
  $I$ and so $\ss \not \in \textrm{TH}_k(I)$. Hence $\textup{TH}_k(I)
  \subseteq \textrm{cl}(\conv(\V_{\RR}(I)))$ and so,
  $\textup{TH}_k(I)$ equals $\textrm{cl}(\conv(\V_{\RR}(I)))$.
\end{proof}

Interestingly, the converse of Lemma~\ref{lem:1ksos implies
  thetakexact} is false in general.

\begin{example} \label{ex:conversefalse} Consider $I = \langle x^2
  \rangle \subset \RR[x]$ with $\V_{\RR}(I) = \{0\} \subset \RR$. All
  linear polynomials that are non-negative on $\V_\RR(I)$ are of the
  form $\pm a^2 x + b^2$ for some $a,b \in \RR$. If $b \neq 0$, then
  $(\pm a^2 x + b^2) \equiv (\frac{a^2}{2b} x \pm b)^2$ mod
  $I$. However, $\pm x$ is not a sum of squares mod $I$, and hence $I$
  is not $(1,k)$-sos for any $k$. On the other hand, $I$ is
  $\textup{TH}_1$-exact since $\textup{conv}(\V_{\RR}(I)) = \{0\}$ is
  cut out by the infinitely many linear inequalities $\pm x + b^2 \geq
  0$ as $b$ varies over $b \neq 0$.
\end{example}

\begin{definition} \label{def:ideal defs} 
  Let $I$ be an ideal in $\RR[\xx]$.   Then $I$ is
\begin{enumerate}
\item {\bf radical} if it equals its {\bf radical ideal}
  $$\sqrt{I} := \{ f \in \RR[\xx] \,:\, f^m \in I, \,\,m \in \NN
  \backslash \{0\} \},$$
\item {\bf real radical} if it equals its {\bf real radical ideal}
  $$\sqrt[\RR]{I} := \{ f \in \RR[\xx] \,:\, f^{2m} + g_1^2 + \cdots +
  g_t^2 \in I, \,\, m \in \NN \backslash \{0\}, \,\, g_1, \ldots,
  g_t \in \RR[\xx] \},$$
\item and {\bf zero-dimensional} if its {\bf complex variety}
$\V_{\CC}(I) := \{\xx \in \CC^n \,:\, f(\xx) = 0 \,\,\forall \,\, f
\in I \}$ is finite. 
\end{enumerate}
\end{definition}

Recall that given a set $S \subseteq \RR^n$, its {\em vanishing ideal}
in $\RR[\xx]$ is the ideal $\I(S) := \{ f \in \RR[\xx] \,:\, f(\ss)=0
\,\,\forall\,\, \ss \in S \}$.  {\em Hilbert's Nullstellensatz} states
that for an ideal $I \subseteq \RR[\xx]$, $\sqrt{I} = \I(\V_{\CC}(I))$
and the {\em Real Nullstellensatz} states that $\sqrt[\RR]{I} =
\I(\V_{\RR}(I))$.  Hence, $I \subseteq \sqrt{I} \subseteq
\sqrt[\RR]{I}$, and if $I$ is real radical then it is also
radical. See for example, \cite[Appendix 2]{MarshallBook}, for these
notions.

We will prove in Section~\ref{sec:theta bodies} that the converse of
Lemma~\ref{lem:1ksos implies thetakexact} holds for real radical
ideals. These ideals occur frequently in applications and for them,
Problem~\ref{prob:lovasz} is asking when $I$ is $\textup{TH}_1$-exact,
or more generally, $\textup{TH}_k$-exact.

\vspace{.2cm}

\noindent{\bf Contents of this paper.}
Recall that a {\em semidefinite program}
(SDP) is an optimization problem in the space of real symmetric
matrices of the form:
\begin{equation} \label{sdp}
\max \,\,\left\{ \cc^t \xx \,:\, A_0 + \sum_{i=0}^m A_i x_i
  \succeq 0 \right \},
\end{equation}
where $\cc \in \RR^m$ and the $A_j$'s are real symmetric matrices. The
notation $A \succeq 0$ implies that $A$ is {\em positive
  semidefinite}. SDPs generalize linear programs and can be solved
efficiently \cite{BoydVandenberghe}.  In Section~\ref{sec:theta
  bodies} we prove that under a certain technical hypothesis
(satisfied by real radical ideals for instance), the theta body
sequence of an ideal $I$ is a modified version of a hierarchy of
relaxations for the convex hull of a basic semialgebraic set, due to
Lasserre \cite{Lasserre1, Lasserre2}.  In this case, each theta body
is the closure of the projection of a {\em spectrahedron} (feasible
region of a SDP), and an explicit representation is possible using the
{\em combinatorial moment matrices} introduced by Laurent
\cite{Laurent}. When $I$ is a real radical ideal, we further prove
that $I$ is $(1,k)$-sos if and only if $I$ is $\textup{TH}_k$-exact
which impacts later sections.

In Section~\ref{sec:examples} we illustrate the theta body sequence
for the maximum stable set and maximum cut problems in a graph which
are two very well-studied problems from combinatorial
optimization. The stable set problem motivated
Problem~\ref{prob:lovasz}. We explain this connection in detail in
Section~\ref{sec:examples}.

In Section~\ref{sec:structure} we solve Problem~\ref{prob:lovasz} for
vanishing ideals of finite point sets in $\RR^n$. This situation
arises often in applications and is the typical set up in
combinatorial optimization. Several corollaries follow: If $S \subset
\RR^n$ is finite and its vanishing ideal $\I(S)$ is $(1,1)$-sos then
$S$ is affinely equivalent to a subset of $\{0,1\}^n$ and its convex
hull can have at most $2^n$ facets. If $S$ is the vertex set of a
down-closed $0/1$-polytope in $\RR^n$, then $\I(S)$ is $(1,1)$-sos if
and only if $\textup{conv}(S)$ is the stable set polytope of a perfect
graph. Families of finite sets in growing dimension with $(1,1)$-sos
vanishing ideals are exhibited.

In Section~\ref{sec:structure arbitrary S}, we give an intrinsic
description of the first theta body, $\textup{TH}_1(I)$, of an
arbitrary polynomial ideal $I$ in terms of the convex quadrics in $I$.
This leads to non-trivial examples of $\textup{TH}_1$-exact ideals
with arbitrarily high-dimensional real varieties and reveals the
algebraic-geometric structure of $\textup{TH}_1(I)$. Analogous
descriptions for higher theta bodies remain open.

\begin{remark} In \cite{Lasserre3}, Lasserre introduced the {\em
    Schm{\"u}dgen Bounded Degree Representation} (S-BDR) and the {\em
    Putinar-Prestel Bounded Degree Representation} (PP-BDR) properties
  of a compact basic semialgebraic set $K = \{\xx \,:\, g_1(\xx) \geq
  0,\ldots,g_m(\xx) \geq 0\}$ (where $g_i \in \RR[\xx]$),
  defined as follows:
\begin{itemize}
\item $K$ has the S-BDR property if there exists a positive integer
  $k$ such that {\em almost all} linear $f$ that are positive over $K$
  has a representation as $f = \sum_{J \subseteq [m]} \sigma_J g_J$
  where $\sigma_J$ are sos, $g_J := \prod_{j \in J} g_j$ and the
  degree of $\sigma_J g_J$ is at most $2k$ for all $J \subseteq [m] :=
  \{1,\ldots,m\}$.
\item $K$ has the PP-BDR property if there exists a positive integer
  $k$ such that {\em almost all} linear $f$ that are positive over $K$
  has a representation as $f = \sum_{j=0}^{m} \sigma_j g_j$ where
  $\sigma_j$ are sos, $g_0 := 1$ and the degree of $\sigma_j g_j$ is
  at most $2k$ for $j=0,\ldots,m$.
\end{itemize}
Call the smallest such $k$ the S-BDR (respectively, PP-BDR) rank of
$K$. Here ``almost all'' means all except a set of Lebesgue measure
zero.  Note that the PP-BDR property implies that S-BDR property.

For an ideal $I=\left< f_1,\ldots,f_m \right> \subset \RR[\xx]$,
$V_\RR(I)$ is the, possibly non-compact, basic semialgebraic set
$\{\xx \in \RR^n \,:\, \pm f_1(\xx) \geq 0, \ldots, \pm f_m(\xx) \geq
0 \}$.  When $\V_\RR(I)$ is compact, its PP-BDR property is closely
related to the $(1,k)$-sos and $\textup{TH}_k$-exact properties of
$I$. However, these notions are not exactly comparable since the
PP-BDR rank of $\V_\RR(I)$ depends on the choice of generators of $I$,
and only the linear polynomials that are positive (as opposed to
non-negative) over $\V_\RR(I)$. Regardless, note that if $\V_{\RR}(I)$
has PP-BDR rank $k$, then $I$ has theta-rank at most $k$.
\end{remark}

{{\bf Acknowledgments}. We thank Monique Laurent and Ting Kei Pong for
  several useful inputs to this paper. We also thank the referees for
  their many constructive comments that helped the organization of the
  paper.}

\section{Theta Bodies} \label{sec:theta bodies}

In Definition~\ref{def:theta} we introduced the $k$-th theta body of a
polynomial ideal $I \subseteq \RR[\xx]$ and observed that these bodies
create a nested sequence of closed convex relaxations of
$\textup{conv}(\V_{\RR}(I))$ with $\textup{TH}_k(I) \supseteq
\textup{TH}_{k+1}(I) \supseteq \textup{conv}(\V_{\RR}(I))$.  Lasserre
\cite{Lasserre1} and Parrilo \cite{Parrilo:phd,Parrilo:spr} have
independently introduced hierarchies of semidefinite relaxations for
polynomial optimization over basic semialgebraic sets in $\RR^n$ using
results from real algebraic geometry and the theory of moments. We
first examine the connection between the theta bodies of an ideal $I$
and Lasserre's relaxations for $\textup{conv}(\V_{\RR}(I))$.


\subsection{Lasserre's hierarchy and theta bodies}

\begin{definition}
  Let $I$ be an ideal in
  $\RR[\xx]$. The {\bf quadratic module} of $I$ is
  $$\mathcal{M}(I):=\left\{s + I \,:\, s
    \textrm{ is sos in} \,\,\RR[\xx] \right \}.$$
  The {\bf $k$-th truncation} of $\mathcal{M}(I)$ is
  $$\mathcal{M}_{k}(I):=\left\{s + I \,:\, s
    \textrm{ is } k\textrm{-sos} \right\}.$$
 \end{definition}
 
 Both $\mathcal{M}(I)$ and $\mathcal{M}_{k}(I)$ are cones in the
 $\RR$-vector space $\RR[\xx]/I$. Let $(\RR[\xx]/I)'$ denote the set
 of linear functionals on $\RR[\xx]/I$ and $\pi_I$ be the projection map
 from $(\RR[\xx]/I)'$ to $\RR^n$ defined as
 $\pi_I(y)=(y(x_1+I),\ldots,y(x_n+I)).$ Also let
 $\mathcal{M}_{k}(I)^{*} \subseteq (\RR[\xx]/I)'$ denote the {\em
   dual cone} to $\mathcal{M}_{k}(I)$, the set of all linear functions
 on $\RR[\xx]/I$ that are non-negative on $\mathcal{M}_{k}(I)$.
 
\begin{definition}
  For $y \in (\RR[\xx]/I)'$, let $H_y$ be the symmetric bilinear
  form
  $$
  \begin{array}{rccc}
    H_y : & \RR[\xx]/I \times \RR[\xx]/I & \longrightarrow & \RR\\
    & (f+I,g+I) & \longmapsto & y(fg+I)
 \end{array}
 $$
 and $H_{y,t}$ be the restriction of $H_y$ to the subspace
 $\RR[\xx]_{t}/I$.
\end{definition}

Recall that a symmetric bilinear form $H:V \times V \rightarrow \RR$,
where $V$ is a $\RR$-vector space, is positive semidefinite
(written as $H \succeq 0$) if $H(v,v) \geq 0$ for all non-zero
elements $v \in V$.  Given a basis $B$ of $V$, the matrix indexed by
the elements of $B$ with $(b_i,b_j)$-entry equal to $H(b_i,b_j)$ is
called the {\em matrix representation} of $H$ in the basis $B$. The
form $H$ is positive semidefinite if and only if its matrix
representation in any basis is positive semidefinite.

\begin{lemma} \label{thm:sdpformulation} Let $I \subseteq \RR[\xx]$ be
  an ideal and $k$ a positive integer. Then
  $$\mathcal{M}_{k}(I)^{*}=\{y \in (\RR[\xx]/I)' :  H_{y,k} \succeq 0\}.$$
\end{lemma}
\begin{proof}
  Note that $y \in \mathcal{M}_{k}(I)^{*}$ if and only if $y(s + I)
  \geq 0$ for all $k$-sos polynomials $s$. By linearity of $y$ this is
  equivalent to $y(h^2 + I) \geq 0$ for all $h \in \RR[\xx]_k$ which
  is the definition of $H_{y,k}$ being positive semidefinite.
\end{proof}

The original Lasserre relaxations in \cite{Lasserre1} approximate
$\conv(S)$ for a basic semialgebraic set $S=\{\xx \in \RR^n \,:\,
g_i(\xx) \geq 0,\, i=1,\ldots,m\}$ by the sets
$$\left \{ (y(x_1),\ldots,y(x_n)) \,:\, y \in \RR[\xx]', \,\, y(1) = 1, \,\,
  y\left(\sum_{i=0}^{m} s_ig_i \right ) \geq 0 \right \}$$ where $s_i$
are sos, $g_0 := 1$ and the degree of $s_ig_i$ is bounded above by
some fixed positive integer. When there are equations among the
$g_i(\xx) \geq 0$, both Lasserre \cite{Lasserre2} (for $0/1$ point
sets) and Laurent \cite{Laurent} (more generally for finite varieties)
propose doing computations mod the ideal generated by the polynomials
defining the equations, to increase efficiency. We adopt this point of
view since in our case, $S = \V_\RR(I)$, is cut out entirely by
equations, and work with the following definition of a Lasserre
relaxation.

\begin{definition} \label{def:lasserre relaxation} Let $I \subseteq
  \RR[\xx]$ be an ideal, $k$ be a positive integer, and
  $\mathcal{Y}_1$ be the hyperplane of all functions $y \in
  (\RR[\xx]/I)'$ such that $y(1+I)=1$.  The {\bf $k$-th modified
    Lasserre relaxation} $Q_{k}(I)$ of $\textup{conv}(\V_{\RR}(I))$ is
  $$Q_{k}(I):=\pi_I(\mathcal{M}_{k}(I)^* \cap \mathcal{Y}_1).$$
\end{definition}

While $\mathcal{M}_{k}(I)^* \cap \mathcal{Y}_1$ is always closed,
$Q_k(I)$ might not be (see Example~\ref{ex:runningex2}). We first note
that $Q_k(I)$ is indeed a relaxation of $\conv(\V_\RR(I))$.

\begin{lemma} \label{lem:QkI contains conv}
For an ideal $I$ and a positive integer $k$,
$\textup{conv}(\V_{\RR}(I)) \subseteq Q_k(I)$. 
\end{lemma}

\begin{proof}
  For $\ss \in \V_{\RR}(I)$, consider $y^{\ss} \in (\RR[\xx]/I)'$
  defined as $y^{\ss}(f+I) := f(\ss)$.  Then $y^{\ss} \in
  \mathcal{M}_{k}(I)^*$ and $y^{\ss}(1+I) = 1$. Therefore,
  $\pi_I(y^{\ss})=\ss \in Q_{k}(I)$, and $\conv(\V_{\RR}(I)) \subseteq
  Q_{k}(I)$ since $Q_{k}(I)$ is convex.
\end{proof}

Since $Q_{k+1}(I) \subseteq Q_{k}(I)$, these bodies create a nested
sequence of relaxations of $\textup{conv}(\V_{\RR}(I)))$ as
intended. Our main goal in this section is to establish a relationship
between $Q_k(I)$ and the $k$-th theta body, $\textup{TH}_k(I)$, of the
ideal $I$ (cf. Definition \ref{def:theta}). We start by noting the
following inclusion.

\begin{proposition} \label{prop:basis free formulation} For an ideal
  $I \subseteq \RR[\xx]$ and a positive integer $k$,
  $\textup{cl}(Q_{k}(I)) \subseteq \textup{TH}_k(I)$.
\end{proposition}
\begin{proof}
  Since $\textup{TH}_k(I)$ is closed, it is enough to show that
  $Q_{k}(I) \subseteq \textup{TH}_k(I)$. Pick $\pp \in Q_{k}(I)$ and
  $y \in \mathcal{M}_{k}(I)^* \cap \mathcal{Y}_1 $ such that
  $\pi_I(y)=\pp$. Let $f = a_0 + \sum_{i=1}^{n} a_i x_i$ and $f+I \in
  \mathcal{M}_{k}(I)$. Then, $\pp \in \textup{TH}_k(I)$ since 
  $$f(\pp) = f(\pi_I(y))= a_0 y(1+I)+ \sum_{i=1}^{n} a_i y(x_i + I) =
  y(f+I) \geq 0.$$
\end{proof}

Theorem~\ref{thm:basis free formulation} will prove that if ${\mathcal
  M}_k(I)$ is closed, we have the equality $\textup{cl}(Q_{k}(I)) =
\textup{TH}_k(I)$.

\begin{lemma} \label{lem:new lemma} Let $I \subseteq \RR[\xx]$ be an
  ideal and $k$ be a positive integer. If $f \in \RR[\xx]_1$ is
  non-negative over $Q_k(I)$, then $f+I \in \textup{cl}({\mathcal
    M}_k(I))$.
\end{lemma} 

\begin{proof} Suppose $f \in \RR[\xx]_1$ is non-negative over $Q_k(I)$
  and $f+I \not \in \textup{cl}({\mathcal M}_k(I))$. Then by the
  separation theorem, there exists $y \in {\mathcal M}_k(I)^*$ such
  that $y(f+I) < 0$. Since $(f+r+I)^2 = (f+r)^2+I$ lies in ${\mathcal
    M}_k(I)$ for any real number $r$, $y \in {\mathcal M}_k(I)^*$ and
  $y$ is linear, we get
$$0 \leq y((f+r+I)^2) = y(f^2+I) + 2ry(f+I) + r^2y(1+I)$$
which implies that $y(1+I) > 0$ since $y(f+I) \neq 0$. Scaling $y$
such that $y(1+I) = 1$, we have that $y \in {\mathcal M}_k(I)^* \cap
{\mathcal Y}_1$. This implies that $\pi_I(y) \in Q_k(I)$ and
therefore, by hypothesis, $f(\pi_I(y)) \geq 0$. However, since $f \in
\RR[\xx]_1$ and $y$ is linear, we also get $f(\pi_I(y)) = y(f+I) < 0$
which is a contradiction.
\end{proof}

\begin{theorem} \label{thm:basis free formulation} Let $I \subseteq
  \RR[\xx]$ be an ideal. For a positive integer $k$, if ${\mathcal
    M}_k(I)$ is closed, then $\textup{cl}(Q_{k}(I)) =
  \textup{TH}_k(I)$.
\end{theorem}

\begin{proof}
  By Proposition~\ref{prop:basis free formulation}, we need to prove
  that when ${\mathcal M}_k(I)$ is closed, $\textup{TH}_k(I) \subseteq
  \textup{cl}(Q_{k}(I))$.  Suppose $\pp \not \in
  \textup{cl}(Q_{k}(I))$. By the separation theorem, there exists $f
  \in \RR[\xx]_1$ non-negative on $\textup{cl}(Q_{k}(I))$ with $f(\pp)
  < 0$. By Lemma~\ref{lem:new lemma}, $f+I \in {\mathcal M}_k(I)$
  since ${\mathcal M}_k(I)$ is closed by assumption and hence $f$ is
  $k$-sos mod $I$. Since $f(\pp) < 0$, $\pp \not \in
  \textup{TH}_k(I)$.
\end{proof}

An important class of ideals for which ${\mathcal M}_k(I)$ is closed
is the set of real radical ideals which are the focus of
Sections~\ref{sec:examples} and ~\ref{sec:structure}. We now derive
various corollaries to Theorem~\ref{thm:basis free formulation} that
apply to real radical ideals.

\begin{corollary} \label{cor:equality for real radical ideals} If $I
  \subseteq \RR[\xx]$ is a real radical ideal then
  $\textup{cl}(Q_{k}(I)) = \textup{TH}_k(I)$.
\end{corollary}

\begin{proof}
By \cite[Prop 2.6]{PowSchei}, if $I$ is real radical, then ${\mathcal
  M}_k(I)$ is closed.
\end{proof}

\begin{lemma} \label{lem:borweinlewis} Let $V$ and $W$ be finite
  dimensional vector spaces, $H \subseteq W$ be a cone and $A \,:\, V
  \rightarrow W$ be a linear map such that $A(V) \cap \textup{int}(H)
  \neq \emptyset$. Then $(A^{-1}H)^* = A'(H^*)$ where $A'$ is the dual
  operator to $A$.  In particular, $A'(H^*)$ is closed in $V'$.
\end{lemma}

\begin{proof}
This follows from Corollary~3.3.13 in \cite{BorweinLewis} by
setting $K = V$.
\end{proof}

\begin{corollary}\label{cor:theta-lasserre}
  Let $I$ be a real radical ideal in $\RR[\xx]$ and $k$ be a positive
  integer.  If there exists $g \in \RR[\xx]_1$ such that $g+I$ is in
  the interior of $\mathcal{M}_{k}(I)$ (considered as a subset of
  $\RR[\xx]_{2k}/I$), then $\textup{TH}_{k}(I)=Q_{k}(I)$.
\end{corollary}

\begin{proof}  
By Corollary~\ref{cor:equality for real radical ideals}, it suffices
to show that $Q_k(I)$ is closed. Consider $\RR^{n+1}$ with coordinates
indexed $0,1,\ldots,n$ and the map
$$\tilde{\pi}_I \,:\, (\RR[\xx]/I)' \rightarrow \RR^{n+1}
\,\,\textup{such that} \,\,y \mapsto (y(1+I),\pi_I(y)).$$ Then
$\tilde{\pi}_I({\mathcal M}_k(I)^* \cap {\mathcal Y}_1) = \{(1,\pp)
\,:\, \pp \in Q_k(I) \}$ and so, $Q_k(I)$ will be closed if
$\tilde{\pi}_I({\mathcal M}_k(I)^* \cap {\mathcal Y}_1) =
\tilde{\pi}_I({\mathcal M}_k(I)^*) \cap \{\pp \in \RR^{n+1} \,:\, p_0
= 1 \}$ is closed.  Hence, it suffices to show that
$\tilde{\pi}_I({\mathcal M}_k(I)^*) \subseteq \RR^{n+1}$ is closed.

Now consider the inclusion map $A \,:\, \RR[\xx]_1/I \rightarrow
\RR[\xx]_{2k}/I$ and let $M$ denote the cone ${\mathcal M}_k(I)$
considered as a subset of $\RR[\xx]_{2k}/I$.  By assumption,
$A(\RR[\xx]_1/I) \cap \textup{int}(M) = (\RR[\xx]_1/I) \cap
\textup{int}(M) \neq \emptyset$ and so by
Lemma~\ref{lem:borweinlewis}, $A'(M^*)$ is closed.  Let $C :=
\{1^*,x_1^*,\ldots,x_n^*\}$ be the canonical basis of
$(\RR[\xx]_1/I)'$. Then $A'(\bar{y}) = ({\bar y}(1+I),{\bar
  y}(x_1+I), \ldots, {\bar y}(x_n+I))$ with respect to $C$. Now note
that if $y \in {\mathcal M}_k(I)^*$ then its restriction ${\hat y}$ to
$\RR[\xx]_{2k}/I$ belongs to $M^*$ and $A'({\hat y}) =
\tilde{\pi}_I(y)$. Therefore, $\tilde{\pi}_I({\mathcal M}_k(I)^*)
\subseteq A'(M^*)$. Conversely, if ${\bar y} \in M^*$ and ${\tilde
  y}$ is any extension to $\RR[\xx]/I$, then ${\tilde y}$ belongs to
${\mathcal M}_k(I)^*$. Since $A'({\bar y}) =
\tilde{\pi}_I(\tilde{y})$ we get $A'(M^*) \subseteq
\tilde{\pi}_I({\mathcal M}_k(I)^*)$ and so, $A'(M^*) =
\tilde{\pi}_I({\mathcal M}_k(I)^*)$ is closed.
\end{proof}


Let $I$ be an ideal and $k$ a positive integer. In
Lemma~\ref{lem:1ksos implies thetakexact} we saw that if $I$ is
$(1,k)$-sos (i.e., every linear polynomial that is non-negative on
$\V_{\RR}(I)$ is $k$-sos mod $I$), then $I$ is $\textup{TH}_k$-exact
(i.e., $\textup{TH}_k(I) =
\textup{cl}(\textup{conv}(\V_{\RR}(I)))$). Example~\ref{ex:conversefalse}
showed that the reverse implication does not always hold.  

\begin{corollary} \label{cor:real radical equivalence} If an ideal $I
  \subseteq \RR[\xx]$ is real radical then $I$ is $(1,k)$-sos if and
  only if $I$ is $\textup{TH}_k$-exact.
\end{corollary}

\begin{proof} 
  If $f \in \RR[\xx]_1$ is non-negative on $\V_\RR(I)$ and $I$ is
  $\textup{TH}_k$-exact, then $f$ is non-negative on
  $\textup{TH}_k(I)$ and hence on $Q_k(I)$. Therefore, by
  Lemma~\ref{lem:new lemma}, $f \in \textup{cl}({\mathcal M}_k(I))$.
  Suppose now that $I$ is also real radical. Then ${\mathcal M}_k(I)$
  is closed and $f+I \in {\mathcal M}_k(I)$, which means that $I$ is
  $(1,k)$-sos.
\end{proof}

We close with a brief discussion of ideals for which the theta body
sequence is guaranteed to converge (finitely or asymptotically) to
$\textup{cl}(\conv(\V_{\RR}(I)))$.

\begin{enumerate}
\item If $\V_{\RR}(I)$ is finite the results in \cite{LaLauRos} imply
  that $I$ is $\textup{TH}_k$-exact for some finite $k$. If
  $\V_\CC(I)$ is finite ($I$ is zero-dimensional), then $k$ can be
  bounded above by the maximum degree of a linear basis of
  $\RR[\xx]/I$ \cite{Laurent} (see Section~\ref{subsec:combinatorial
    matrices}).  However, as in $I = \langle x^2 \rangle$, we cannot
  guarantee that $I$ is $(1,k)$-sos for any $k$, even when $I$ is
  zero-dimensional.  If $I$ is zero-dimensional and radical, then in
  fact, $I$ is $(1,k)$-sos for finite $k$ with $k \leq
  |\V_{\CC}(I)|-1$ (see \cite{parrilo}, \cite[Theorem
  2.4]{LaurentSosSurvey}). Better bounds are often possible as in
  Remark~\ref{rem:k parallel translates}. For an ideal $I \subseteq
  \RR[\xx]$ we summarize the above results in the following table.


\begin{center}
$
\xymatrix{
\V_\CC(I) \,\,\text{finite}  \ar[r] \ar[d]^{\tiny{I = \sqrt{I}}}  & \V_\RR(I) \,\,\text{finite} \ar[d] \\
I \,\,\, \text{(1,k)-sos}  \ar@<-1mm>[r]  &  I\,\,\,\textup{TH}_k\textup{-exact}   \ar@<-1mm>[l]_{\tiny{I=\sqrt[\RR]{I}}}
}
$
\end{center}

\item If $\V_{\RR}(I)$ is not finite but is compact, Schm{\"u}dgen's
  Positivstellensatz \cite[Chapter 3]{MarshallBook} implies that the
  theta body sequence of $I$ converges (at least asymptotically) to
  $\textup{cl}(\textup{conv}(\V_{\RR}(I)))$ (i.e.,
  $\bigcap_{k=1}^{\infty} \textup{TH}_k(I) =
  \textup{cl}(\textup{conv}(\V_{\RR}(I)))$).  
  
\item If $\V_{\RR}(I)$ is not compact, then the study of the theta
  body hierarchy becomes harder.  Scheiderer \cite[Chapter
  2]{MarshallBook} has identified ideals $I$ with $\V_{\RR}(I)$ not
  necessarily compact, but of dimension at most two, for which every
  $f \geq 0$ mod $I$ is sos mod $I$. In all these cases, the theta
  body sequence of $I$ converges to
  $\textup{cl}(\textup{conv}(\V_{\RR}(I)))$.
  \end{enumerate}

  The results of Schm{\"u}dgen and Scheiderer mentioned above fit
  within a general framework in real algebraic geometry that is
  concerned with when an arbitrary $f \in \RR[\xx]$ that is positive
  or non-negative over a basic semi-algebraic set is sos modulo
  certain algebraic objects defined by the set. We only care about
  real varieties and whether linear polynomials that are non-negative
  over them are sos mod their ideals. Therefore, often there are
  ideals $I$ that are $\textup{TH}_k$-exact or $(1,k)$-sos for which
  there are non-linear polynomials $f$ such that $f \geq 0$ mod $I$
  but $f$ is not sos mod $I$. For instance, the proof of
  Theorem~\ref{thm:intersection} will show that $J_n := \langle
  \sum_{i=1}^{n} x_i^2 -1 \rangle$ is $(1,1)$-sos for all $n$, but a
  result of Scheiderer \cite[Theorem 2.6.3]{MarshallBook} implies that
  when $n \geq 4$, there is always some non-linear $f$ non-negative on
  $\V_\RR(J_n)$ that is not sos mod $J_n$.

\subsection{Combinatorial moment matrices} \label{subsec:combinatorial
  matrices}

To compute theta bodies we must work with the truncated quadratic
module $\mathcal{M}_{k}(I)$ which requires computing sums of squares
in $\RR[\xx]/I$ as described in \cite{SOSstructbook}, or dually, using
the combinatorial moment matrices introduced by Laurent in
\cite{Laurent}.  We describe the latter viewpoint here as it is more
natural for theta bodies. 

Consider a basis $\B = \{ f_0+I, f_1+I, \ldots \}$ for $\RR[\xx]/I$,
and define $\textup{deg}(f_i + I) := \textup{min}_{f-f_i \in I}
\textup{deg}\,f$.  For a positive integer $k$, let $\B_k := \{ f_l + I
\in \B \,:\, \textup{deg}(f_l+I) \leq k \}$, and set $\ff_k := (f_l+I
\,:\, f_l+I \in \B_k)$. We may assume that the elements of $\B$ are
indexed in order of increasing degree.  Let $\lambda^{(g+I)} :=
(\lambda_l^{(g+I)})$ be the vector of coordinates of $g+I$ with
respect to $\B$. Note that $\lambda^{(g+I)}$ has only finitely many
non-zero coordinates.

\begin{definition}
  Let $\yy \in \RR^{\mathcal{B}}$. Then the {\bf combinatorial moment
    matrix} $M_{\mathcal{B}}(\yy)$ is the (possibly infinite) matrix
  indexed by $\mathcal{B}$ whose $(i,j)$ entry is
  $$\lambda^{(f_i f_j +I)} \cdot \yy = \sum \lambda_l^{(f_if_j+I)}
  y_l.$$
  The {\bf $k$-th-truncated combinatorial moment matrix}
  $M_{\B_k}(\yy)$ is the finite (upper left principal) submatrix of
  $M_{\B}(\yy)$ indexed by $\B_k$.
\end{definition}

Although only a finite number of the components in
$\lambda^{(f_if_j+I)}$ are non-zero, for practical purposes we need to
control exactly which indices can be non-zero.  One way to do this is
by choosing $\B$ such that if $f+I$ has degree $k$ then $f+I \in
\textup{span}(\B_k)$. This is true for instance if $\B$ is the set of
{\em standard monomials} of a {\em term order} that respects degree
\cite{CLO}. If $\B$ has this property then $M_{\B_k}(\yy)$ only
depends on the entries of $\yy$ indexed by $\B_{2k}$. 

\begin{theorem}\label{thm:combinatorial matrices formulation}
  For each positive integer $k$, 
  $$\textup{proj}_{\RR^{\B_1}}\{ \yy \in \RR^{\B_{2k}} \,:\,
  M_{\B_k}(\yy) \succeq 0, \, y_0 = 1\} = \ff_1 (Q_{k}(I)),$$ where
  $y_0$ is the first entry of $\yy \in \RR^{B_{2k}}$,
  $\textup{proj}_{\RR^{\B_1}}$ is the projection onto the coordinates
  indexed by $\B_1$, and for $\pp \in \RR^n$, $\ff_1(\pp) :=
  (f_i(\pp))_{f_i+I \in \B_1}$.
\end{theorem}
\begin{proof}
  We may identify $\yy =(y_i) \in \RR^{\B_{2k}}$ with the operator
  $\bar{y} \in (\RR[\xx]/I)'$ where $\bar{y}(f_i+I)=y_i$ if $f_i+I
  \in \B_{2k}$ and zero otherwise. Then $M_{\B_k}(\yy)$ is simply the
  matrix representation of $H_{\bar{y},k}$ in the basis $\B$, since
  we assumed that if $\deg \,(f_i+I), \deg \,(f_j
  + I) \leq k$ then $\bar{y}(f_if_j+I)$ depends only on the value of
  $\bar{y}$ on $\B_{2k}$. Therefore, $\textup{proj}_{\RR^{\B_1}}\{ \yy
  \in \RR^{\B_{2k}} \,:\, M_{\B_k}(\yy) \succeq 0\}$
  equals $$\{(\bar{y}(f_i+I))_{\B_1} : \bar{y} \in (\RR[\xx]/I)',
  H_{\bar{y},k} \succeq 0\}.$$ 
  Furthermore, since $f_i$ is linear whenever $f_i+I \in \B_1$,
  $$(\bar{y}(f_i+I))_{\B_1}= (f_i(\pi_I(\bar{y})))_{\B_1} =:
  \ff_1(\pi_I(\bar{y}))$$ so by Lemma~\ref{thm:sdpformulation},
  $\textup{proj}_{\RR^{\B_1}}\{ \yy \in \RR^{\B_{2k}} \,:\,
  M_{\B_k}(\yy) \succeq 0, \, y_0 = 1\} = \ff_1 (Q_{k}(I)$.
\end{proof}

\begin{corollary} \label{cor:simple version} Suppose $\B_1 = \{1+I,
  x_1+I, \ldots, x_n+I\}$ and denote by $y_0,y_1, \ldots, y_n$ the
  first $n+1$ coordinates of $\yy \in \RR^{\B_{2k}}$, then $$Q_{k}(I)
  = \{ (y_1, \ldots, y_n) \,:\, \yy \in
  \RR^{\B_{2k}}\,\textup{with}\,\, M_{\B_k}(\yy) \succeq 0
  \,\textup{and} \, y_0 = 1\}.$$
\end{corollary}

By Corollary \ref{cor:simple version}, optimizing a linear function
over $Q_k(I)$, hence over $\textup{cl}(Q_k(I))$, is
an SDP and can be solved efficiently.

\begin{example} \label{ex:runningex2} Consider the ideal $I = \langle
  x_1^2x_2 - 1 \rangle \subset \RR[x_1,x_2]$ from
  Example~\ref{ex:runningex} for which $\textup{conv}(\V_{\RR}(I)) =
  \{(s_1,s_2) \in \RR^2 \,:\, s_2 > 0 \}$ was not closed but $I$ was
  $\textup{TH}_2$-exact and $(1,2)$-sos. Note that $\B = \bigcup_{k
    \in \NN} \{x_1^k+I, x_2^k+I, x_1x_2^k+I\}$ is a degree-compatible
  monomial basis for $\RR[x_1,x_2]/I$ for which $$\B_4 =
  \{1,x_1,x_2,x_1^2,x_1x_2,x_2^2,x_1x_2^2,x_1^3,x_2^3,x_1x_2^3,x_1^4,x_2^4
  \}+I.$$ The combinatorial moment matrix $M_{\B_2}(\yy)$ for
  $\yy=(1,y_1,\ldots, y_{11}) \in \RR^{\B_4}$ is
  $$ \begin{array}{ll}
     & \hspace*{.2cm} \begin{array}{cccccc} 1 & \, x_1 & x_2 &
     x_1^2 & x_1x_2 & \hspace*{-.1cm} x_2^2 
     \end{array} \\ 
\vspace*{-.2cm}
& \\ 
\begin{array}{c}
    1 \\ x_1 \\ x_2 \\ x_1^2 \\ x_1x_2 \\ x_2^2 
     \end{array} & \hspace*{-.2cm}
\left( \begin{array}{cccccc}
1 & y_1 & y_2 & y_3 & y_4 & y_5 \\
y_1 & y_3 & y_4 & y_6 & 1 & y_7 \\
y_2 & y_4 & y_5 & 1 & y_7 & y_8 \\
y_3 & y_6 & 1 & y_9 & y_1 & y_2 \\
y_4 & 1 & y_7 & y_1 & y_2 & y_{10}\\
y_5 & y_7 & y_8 & y_2 & y_{10} & y_{11}
\end{array} \right)
\end{array}
$$
If $M_{\B_2}(\yy) \succeq 0$, then the principal minor indexed by
$x_1$ and $x_1x_2$ implies that $y_2y_3 \geq 1$ and so in particular,
$y_2 \neq 0$ for all $\yy \in Q_{2}(I)$. However, since $Q_{2}(I)
\supseteq \textup{conv}(\V_{\RR}(I)) = \{(s_1,s_2) \in \RR^2 \,:\, s_2
> 0 \}$, it must be that $Q_{2}(I) = \textup{conv}(\V_{\RR}(I))$ which
shows that $Q_{2}(I)$ is not closed.
\end{example}

\begin{remark}\label{rem:running example}
  Example \ref{ex:runningex2} can be modified to show that $Q_k(I)$
  may not be closed even if $\V_{\RR}(I)$ is finite. To see this,
  choose sufficiently many pairs of points $(\pm t,1/t^2)$ on the
  curve $x_1^2x_2=1$ to form a set $S$ such that the ideal $\I(S)$ has a
  monomial basis $\B'$ in which $\B'_4$ equals the $\B_4$ from
  above. For instance, $S = \{(\pm t, 1/t^2) \,:\, t=1,\ldots,7 \}$
  will work. Then $Q_{2}(\I(S))$ coincides with $Q_{2}(I)$ computed
  above and so is not a closed set.
\end{remark}

We now show that in the particular case of vanishing ideals of $0/1$
points, which are real radical ideals, the closure in
Theorem~\ref{thm:basis free formulation} ($\textup{TH}_k(I) =
\textup{cl}(Q_k(I))$) is not needed. Most ideals that occur in
combinatorial optimization have this form and we will see important
examples in Section~\ref{sec:examples}. Remark \ref{rem:running
  example} shows that the closure cannot be removed for arbitrary
finite point sets.

\begin{proposition} \label{prop:theta=Q} If $S$ is a set of $0/1$
  points in $\RR^n$ and $I=\I(S)$ then for all positive integers $k$,
  $\textup{TH}_k(I)=Q_{k}(I)$.
\end{proposition}

\begin{proof}
  By Corollary~\ref{cor:theta-lasserre} it is enough to show that
  there is a linear polynomial $g \in \RR[\xx]$ such that $g \equiv
  \ff_k^t A \ff_k$ mod $I$ for a {\em positive definite} matrix $A$
  and some basis of $\RR[\xx]/I$ with respect to which $\ff_k$ was
  determined.  Let $\B$ be a monomial basis for $\RR[\xx]/I$ and
  $\B_k=\{1,p_1,\ldots,p_l\}+I$. Let $\cc \in \RR^{l}$ be the vector
  with all entries equal to $-2$, and $D \in \RR^{l \times l}$ be the
  diagonal matrix with all diagonal entries equal to $4$. Since $x_i^2
  \equiv x_i$ mod $I$ for $i=1,\ldots,n$ and $\B$ is a monomial basis,
  for any $f+I \in \B$, $f \equiv f^2$ mod $I$. Therefore, the
  constant
$$l+1 \equiv \ff_k^t
\left[
\begin{array} {cc}
l+1 & \cc^t \\
\cc & D
\end{array}
\right]\ff_k \,\,\textup{mod} \,\,I,$$ and it is enough to prove that
the square matrix on the right is positive definite. This follows from
the fact that $D$ is positive definite and its {\em Schur
  complement} $(l+1) - \cc^t D^{-1} \cc = 1$ is positive
(\cite[Theorem 7.7.6]{HornJohnson}).
\end{proof}

\section{Combinatorial Examples} \label{sec:examples}

An important area of application for the theta body hierarchy
constructed in Section~\ref{sec:theta bodies} is {\em combinatorial
  optimization} which is typically concerned with optimizing a linear
function over a finite set of integer points. In this section, we
compute theta bodies for two important problems in combinatorial
optimization -- the {\em maximum stable set problem} and the {\em
  maximum cut problem} in a graph.  We explain the observations about
the stable set problem which motivated Lov{\'a}sz to pose
Problem~\ref{prob:lovasz}. The cut problem is modeled in two different
ways. The first is a non-standard approach which is described
fully. For the second, more standard model of the cut problem, theta
bodies provide a new hierarchy of semidefinite relaxations for the
{\em cut polytope} that is studied in detail in \cite{GLPT}. We
outline those results briefly here.  A recent trend in theoretical
computer science has been to study the computational complexity of
approximating problems in combinatorial optimization via the standard
hierarchies of convex relaxations to these problems such as those in
\cite{LovaszSchrijver91} and \cite{Lasserre1, Lasserre2}. Our theta
body approach provides a new mechanism to establish such complexity
results.

\subsection{The Maximum Stable Set Problem}\label{subsec:stable sets}
Let $G=([n],E)$ be an undirected graph with vertex set
$[n]=\{1,\ldots,n\}$ and edge set $E$. A {\em stable set} in $G$ is a
set $U \subseteq [n]$ such that for all $i,j \in U$, $\{i,j\} \not \in
E$. The maximum stable set problem seeks the stable set of largest
cardinality in $G$, the size of which is the {\em stability number} of
$G$, denoted as $\alpha(G)$. 

The maximum stable set problem can be modeled as follows. For each
stable set $U \subseteq [n]$, let $\chi^U \in \{0,1\}^n$ be its {\em
  characteristic vector} defined as $(\chi^U)_i = 1$ if $i \in U$ and
$(\chi^U)_i = 0$ otherwise.  Let $S_G \subseteq \{0,1\}^n$ be the set
of characteristic vectors of all stable sets in $G$. Then
$\textup{STAB}(G) := \conv(S_G)$ is called the {\em stable set
  polytope} of $G$ and the maximum stable set problem is, in theory,
the linear program $\textup{max}\{ \sum_{i=1}^{n} x_i \,:\, \xx \in
\textup{STAB}(G) \}$ with optimal value $\alpha(G)$. However,
$\textup{STAB}(G)$ is not known apriori, and so one resorts to
relaxations of it over which one can optimize $\sum_{i=1}^{n} x_i$.

In \cite{ShannonCapacity}, Lov{\'a}sz introduced, $\textup{TH}(G)$, a
convex relaxation of $\textup{STAB}(G)$, called the {\em theta body}
of $G$. The problem $\textup{max}\{ \sum_{i=1}^{n} x_i \,:\, \xx \in
\textup{TH}(G) \}$ is a SDP which can be solved to arbitrary precision
in polynomial time in the size of $G$. The optimal value of this SDP
is called the {\em theta number} of $G$ and provides an upper
bound on $\alpha(G)$. See \cite[Chapter 9]{GLS} and
\cite{BruceShepherd} for more on the stable set problem and
$\textup{TH}(G)$.  The body $\textup{TH}(G)$ was the first example of
a SDP relaxation of a discrete optimization problem and snowballed the
use of SDP in combinatorial optimization. See \cite{LaurentRendl,
  Lovasz} for surveys. Recall that a graph $G$ is {\em perfect} if and
only if $G$ has no induced odd cycles of length at least five or their
complements. Lov{\'a}sz showed that $\textup{STAB}(G) =
\textup{TH}(G)$ if and only if $G$ is perfect. This equality shows
that the maximum stable set problem can be solved in polynomial time
in the size of $G$ when $G$ is a perfect graph, and this geometric
proof is the only one known for this complexity result.

The theta body $\textup{TH}(G)$ has many definitions (see
\cite[Chapter 9]{GLS}) but the one relevant for this paper was
observed by Lov{\'a}sz and appears without proof in
\cite{LovaszStablesetsPolynomials}. Let $I_G := \langle x_j^2 - x_j
\,\,\forall\,\,j \in [n], \,\,\, x_ix_j \,\, \forall \,\, \{i,j\} \in
E \rangle \subseteq \RR[\xx]$. Then check that $\V_{\RR}(I_G) = S_G$
and that $I_G$ is both zero-dimensional and real radical.  Lov{\'a}sz
observed that
\begin{equation} \label{lovasz observation}
\textup{TH}(G) = \{ \xx \in \RR^n \,:\, f(\xx) \geq 0  \,\,\forall
 \,\,\textup{ linear } \,\,f\,\,\textup{that is $1$-sos mod} \,\,I_G
 \}.
\end{equation}
By Definition~\ref{def:theta} (1), $\textup{TH}(G)$ is exactly the
first theta body, $\textup{TH}_1(I_G)$, of the ideal $I_G$, and by the
above discussion, $I_G$ is $\textup{TH}_1$-exact (i.e.,
$\textup{TH}_1(I_G) = \textup{STAB}(G)$) if and only if $G$ is
perfect. Lov{\'a}sz observed that, in fact, $I_G$ is $(1,1)$-sos if
and only if $G$ is perfect which motivated Problem~\ref{prob:lovasz}
that asks for a characterizations of all $(1,1)$-sos ideals in
$\RR[\xx]$.  Lov{\'a}sz refers to a $(1,1)$-sos ideal as a {\em
  perfect ideal}. A $(1,1)$-sos ideal $I$ would have the property that
its first and simplest theta body, $\textup{TH}_1(I)$, coincides with
$\textup{cl}(\conv(\V_\RR(I)))$ which is a valuable property for
linear optimization over $\conv(V_\RR(I))$, especially when
$\textup{TH}_1(I)$ is computationally tractable.

The theta body hierarchy of the ideal $I_G$ therefore naturally
extends the theta body of $G$ to a family of nested relaxations of
$\textup{STAB}(G)$. Further, the connection between $\textup{TH}(G)$
and sums of squares polynomials motivated Definition~\ref{def:theta}
which extends the construction of $\textup{TH}(G)$ to a hierarchy of
relaxations of $\V_\RR(I)$ for any ideal $I \subseteq \RR[\xx]$. We
now explicity describe the $k$-th theta body of $I_G$ in terms of
combinatorial moment matrices.

For $U \subseteq [n]$, let $\xx^U := \prod_{i \in U} x_i$. From the
generators of $I_G$ it is clear that if $f \in \RR[\xx]$, then $f
\equiv g$ mod $I_G$ where $g$ is in the $\RR$-span of the set of
monomials $\{ \xx^U \,:\, U \textup{ is a stable set in } G \}$. Check
that $\B := \{\xx^U + I_G \,:\, U \,\,\textup{stable set in} \,\, G\}$
is a basis of $\RR[\xx]/I_G$ containing $1+I_G, x_1 + I_G, \ldots, x_n
+ I_G$. Therefore, by Corollary~\ref{cor:simple version} and
Proposition~\ref{prop:theta=Q} we have
$$\textup{TH}_k(I_G) = \left\{ \yy \in \RR^n \,:\,
\begin{array}{l} 
  \exists \, M \succeq 0, \, M \in \RR^{|\B_k| \times |\B_k|}
  \,\textup{such that} \\ 
  M_{\emptyset \emptyset} = 1,\\
  M_{\emptyset \{i\}} = M_{\{i\} \emptyset} = M_{\{i\} \{i\}} = y_i \\
  M_{U U'} = 0 \,\,\textup{if} \,\,U \cup U' \,\, \textup{is not stable
    in}  \,\, G\\
  M_{U U'} = M_{W W'} \,\,\textup{if} \,\, U \cup U' = W \cup W' 
\end{array}
\right \}.$$ 
In particular,  indexing the one element stable sets by the vertices
of $G$,
$$\textup{TH}_1(I_G) = \left\{ \yy \in \RR^n \,:\,
\begin{array}{l} 
\exists \, M \succeq 0, M \in \RR^{(n+1) \times (n+1)}\,\textup{such that} \\ 
M_{00} = 1,\\
M_{0i} = M_{i0} = M_{ii} =  y_i \,\,\forall\,\,i \in [n]\\
M_{ij} = 0 \,\,\forall \,\, \{i,j\} \in E
\end{array}
\right \}.$$

This description of $\textup{TH}_1(I_G)$ coincides with the
semidefinite description of $\textup{TH}(G)$ (see \cite[Lemma
2.17]{LovaszSchrijver91} for instance) and so, $\textup{TH}(G) =
\textup{TH}_1(I_G)$.  Corollary~\ref{cor:real radical equivalence}
confirms Lov{\'a}sz's observation and adds to his other
characterizations of a perfect graph as follows.

\begin{theorem} \cite[Chapter 9]{GLS}  \label{thm:knownforperfectgraphs}
  The following are equivalent for a graph $G$.
\begin{enumerate}
\item $G$ is perfect.
\item $\textup{STAB}(G) = \textup{TH}(G)$. 
\item $\textup{TH}(G)$ is a polytope. 
\item The complement $\overline G$ of $G$ is perfect.
\item $I_G$ is $(1,1)$-sos.
\end{enumerate}
\end{theorem}

The usual Lasserre relaxations of the maximum stable set problem are
set up from the following initial linear programming relaxation of
$\textup{STAB}(G)$:
$$\textup{FRAC}(G) := \{ \xx \in \RR^n \,:\, x_i \geq 0 \,\,\forall
\,\,i \in [n], \, 1-x_i-x_j \geq 0 \,\,\forall \,\,\{i,j\} \in E \}.$$
Note that $S_G = \textup{FRAC}(G) \cap \{0,1\}^n$.  The $k$-th
Lasserre relaxation of $\textup{STAB}(G)$ (see \cite{Lasserre2},
\cite{LaurentComparisonpaper}) uses both the ideal $\langle x_i^2-x_i
\,:\, i \in [n] \rangle$ and the inequality system describing
$\textup{FRAC}(G)$, whereas in the theta body formulation,
$\textup{TH}_k(I_G)$, there is only the ideal $I_G$ and no
inequalities.  Despite this difference, \cite[Lemma
20]{LaurentComparisonpaper} proves that the usual Lasserre hierarchy
is exactly our theta body hierarchy for the stable set problem.  This
interpretation of the Lasserre hierarchy provides new tools to
understand these relaxations such as establishing the validity of
inequalities over them as shown below.

Since no monomial in the basis $\B$ of $\RR[\xx]/I_G$ has degree
larger than $\alpha(G)$, for any $G$, $I_G$ is $(1,\alpha(G))$-sos and
$\textup{STAB}(G) = \textup{TH}_{\alpha(G)}(I_G)$. However, for many
non-perfect graphs the theta-rank of $I_G$ can be a lot smaller than
$\alpha(G)$.  For instance if $G$ is a $(2k+1)$-cycle, then $\alpha(G)
= k$ while Proposition~\ref{prop:2sos} below shows that the theta-rank
of $I_G$ is two.

\begin{theorem} \label{thm:stabg for oddholes} 
  \cite[Corollary 65.12a]{SchrijverB} If $G=([n],E)$ is an odd cycle
  with $n \geq 5$, then $\textup{STAB}(G)$ is determined by the
  following inequalities:
  $$
  x_i \geq 0 \,\,\forall \,\, i \in [n], \,\,\, 1 - \sum_{i \in K}
  x_i \geq 0 \,\,\forall \textup{ cliques } \,\,K \textup{ in } G, 
  \,\,\, \alpha(G) - \sum_{i \in [n]} x_i \geq 0.
$$
\end{theorem}

\begin{proposition} \label{prop:2sos}
  If $G$ is an odd cycle with at least five vertices, then $I_G$ is
  $(1,2)$-sos and therefore, $\textup{TH}_2$-exact.
\end{proposition}

\begin{proof} 
  Let $n = 2k+1$ and $G$ be an $n$-cycle. Then $I_G = \langle
  x_i^2-x_i, \,\, x_ix_{i+1} \,\,\forall \,\, i \in [n] \rangle$ where
  $x_{n+1} = x_1$.  Therefore, $(1-x_i)^2 \equiv 1-x_i$ and
  $(1-x_i-x_{i+1})^2 \equiv 1 -x_i-x_{i+1} \,\,\textup{mod}\,\,I_G$.
  This implies that, mod $I_G$,
$$p_i^2:=((1-x_1)(1-x_{2i}-x_{2i+1}))^2 \equiv p_i = 1 - x_1
-x_{2i}-x_{2i+1} + x_1x_{2i} + x_1x_{2i+1}.$$
Summing over $i=1,..,k$, we get 
$$\sum_{i=1}^{k} p_i^2 \equiv k - kx_1 -\sum_{i=2}^{2k+1}x_i +
\sum_{i=3}^{2k} x_1x_i \,\,\textup{mod}\,\,I_G$$
since $x_1x_2$ and $x_1x_{2k+1}$ lie in
$I_G$.  Define $g_i := x_1(1-x_{2i+1}-x_{2i+2})$. Then $g_i^2 - g_i
\in I_G$ and mod $I_G$ we get that
$$\sum_{i=1}^{k-1} g_i^2 \equiv (k-1)x_1 - \sum_{i=3}^{2k} x_1x_i,
\,\,\textup{which implies} \,\, \sum_{i=1}^{k} p_i^2 +
\sum_{i=1}^{k-1} g_i^2 \equiv k - \sum_{i=1}^{2k+1}x_i.$$ 

To prove that $I_G$ is $(1,2)$-sos it suffices to show that the left
hand sides of the inequalities in the description of
$\textup{STAB}(G)$ in Theorem~\ref{thm:stabg for oddholes} are $2$-sos
mod $I_G$ since by Farkas Lemma \cite{Sch}, all other linear
inequalities that are non-negative over $S_G$ are non-negative real
combinations of a set of inequalities defining $\textup{STAB}(G)$.
Clearly, $x_i \equiv x_i^2 \,\,\textup{mod}\,\,I_G$ for all $i \in
[n]$ and one can check that for each clique $K$, $(1 - \sum_{i \in K}
x_i) \equiv (1 - \sum_{i \in K} x_i)^2 \,\,\textup{mod}\,\,I_G$. The
previous paragraph shows that $k - \sum_{i=1}^{2k+1}x_i$ is also
$2$-sos mod $I_G$.
\end{proof}

An induced odd cycle $C_{2k+1}$ in $G$, yields the well-known {\em odd
  cycle inequality} $\sum_{i\in C_{2k+1}} x_i \leq \alpha(C_{2k+1})=k$
that is satisfied by $S_G$ \cite[Chapter
9]{GLS}. Proposition~\ref{prop:2sos} implies that for any graph $G$,
$\textup{TH}_2(I_G)$ satisfies all odd cycle inequalities from $G$
since every stable set $U$ in $G$ restricts to a stable set in an
induced odd cycle in $G$. This general result can also be proved using
results from \cite{LovaszSchrijver91} and
\cite{LaurentComparisonpaper}. The direct arguments used in the proof
of Proposition~\ref{prop:2sos} are examples of the algebraic {\em
  inference rules} outlined by Lov{\'a}sz in
\cite{LovaszStablesetsPolynomials}. Similarly, one can also show that
other well-known classes of inequalities such as the {\em odd
  antihole} and {\em odd wheel} inequalities \cite[Chapter 9]{GLS} are
also valid for $\textup{TH}_2(I_G)$. Schoenebeck \cite{Schoenebeck}
has recently shown that there is no constant $k$ such that
$\textup{STAB}(G) = \textup{TH}_k(I_G)$ for all graphs $G$ (as
expected, unless P=NP).  However, no explicit family of graphs that
exhibit this behaviour is known.


\subsection{Cuts in graphs} \label{subsec:cuts} Given an undirected
connected graph $G = ([n], E)$ and a partition of its vertex set $[n]$
into two parts $V_1$ and $V_2$, the set of edges $\{i,j\} \in E$ such
that exactly one of $i$ or $j$ is in $V_1$ and the other in $V_2$ is
the {\em cut} in $G$ induced by the partition $(V_1,V_2)$. The cuts in
$G$ are in bijection with the $2^{n-1}$ distinct partitions of $[n]$
into two sets. The maximum cut problem in $G$ seeks the cut in $G$ of
largest cardinality.  This problem is NP-hard and has received a great
deal of attention in the literature. A celebrated result in this area
is an approximation algorithm for the max cut problem, due to Goemans
and Williamson \cite{GoemansWilliamson}, that guarantees a cut of size
at least $0.878$ of the optimal cut. It relies on a simple SDP
relaxation of the problem.

We first study a non-standard model of the max cut problem. Let
$$SG := \{ \chi^F \,:\, F \subseteq E\,\,\textup{is contained in a cut
  of} \,\, G \} \subseteq \{0,1\}^E.$$ Then the {\em weighted} max cut
problem with non-negative weights $w_e$ on the edges $e \in E$ is
$\textup{max} \left \{ \sum_{e \in E} w_e x_e \,:\, \xx \in SG \right
\}$, and the vanishing ideal
$$\I(SG) = \langle x_e^2 - x_e \,\,, \xx^T \,:\,
\,\,e \in E, \,\, T \textup{ odd cycle in } G \rangle. $$ A
basis of $\RR[\xx]/{\I(SG)}$ is $$\B = \{ \xx^U + I(SG) \,:\, U
\subseteq E \,\,\textup{does not contain an odd cycle in} \,\, G \}$$
and $1+\I(SG), x_e + \I(SG) \,(\forall \,e \in E)$ lie in
$\B$. Therefore,
$$\textup{TH}_k(\I(SG)) = \left\{ \yy \in \RR^E \,:\,
\begin{array}{l} 
\exists \, M \succeq 0, \, M \in \RR^{|\B_k| \times |\B_k|}
  \,\textup{such that} \\ 
M_{\emptyset \emptyset} = 1,\\
M_{\emptyset \{i\}} = M_{\{i\} \emptyset} = M_{\{i\}\{i\}} = y_i \\
M_{U U'} = 0 \,\,\textup{if} \,\,U \cup U' \,\, \textup{has an odd
  cycle} \\
M_{U U'} = M_{W W'} \,\,\textup{if} \,\, U \cup U' = W \cup W' 
\end{array}
\right \}.$$

In particular,
$$\textup{TH}_1(\I(SG)) = \left\{ \yy \in \RR^E \,:\,
 \begin{array}{l} 
 \exists \, M \succeq 0, \, M \in \RR^{(|E|+1) \times (|E|+1)}
   \,\textup{such that} \\ 
 M_{00} = 1,\\
 M_{0 e} = M_{e 0} = M_{e e} =  y_e \,\,\forall \,\, e \in E
 \end{array}
\right \}.$$

Note that for any graph $G$, $\textup{TH}_1(\I(SG))$ is the unit cube
in $\RR^{E}$ which may not be equal to $\textup{conv}(SG)$. This
stands in contrast to the case of stable sets for which
$\textup{TH}_1(I_G)$ is a polytope if and only if $\textup{TH}_1(I_G)
= \textup{STAB}(G)$. 

\begin{proposition}
  The ideal $\I(SG)$ is $\textup{TH}_1$-exact if and only if $G$ is a
  bipartite graph.
\end{proposition}

\begin{proof} This follows immediately from the description of
  $\textup{TH}_1(\I(SG))$ and from the fact that $G$ is bipartite if
  and only if it has no odd cycles.
\end{proof}

Since the maximum degree of a monomial in $\B$ is the size of the max
cut in $G$, the theta-rank of $\I(SG)$ is bounded from above by the
size of the max cut in $G$. 

\begin{proposition} \label{prop:increasingk} There is no constant $k$
  such that $\I(SG)$ is $\textup{TH}_k$-exact for all graphs $G$.
\end{proposition}

\begin{proof}
  Let $G$ be a $(2k+1)$-cycle. Then $\textup{TH}_k(\I(SG)) \neq
  \textup{conv}(SG)$ since the linear constraint imposed by the cycle
  in the definition of $\textup{TH}_k(\I(SG))$ will not appear in
  theta bodies of index $k$ or less. 
\end{proof}


The theta bodies of a second, more standard, formulation of the
weighted max cut problem are studied in \cite{GLPT}. In this setup,
each cut $C$ in $G=([n],E)$ is recorded by its {\em cut vector}
$\chi^C \in \{\pm 1\}^{E}$ with $\chi^C_{\{i,j\}} = 1$ if $\{i,j\}
\not \in C$ and $\chi^C_{\{i,j\}} = -1$ if $\{i,j\} \in C$.  Let $E_n$
denote the edge set of the complete graph $K_n$, and $\pi_E$ be the
projection from $\RR^{E_n}$ to $\RR^E$. The {\bf cut polytope} of $G$
is
$$\textup{CUT}(G) := \textup{conv}\{ \chi^C : C \textrm{ is a cut in }
G\} \subseteq \RR^{E} = \pi_E(\textup{CUT}(K_n)),$$
and the weighted
max cut problem, for weights $w_e \in \RR$ ($\forall$ $e \in E$) becomes
 $$
 \max \left\{\frac{1}{2}\sum_{e \in E} w_{e}(1- x_{e}) : \xx \in
   \textup{CUT}(G)\right\}.$$

In \cite{GLPT}, the vanishing ideal $IG$ of the cut vectors $\{ \chi^C
\,:\, C \textup{ is a cut in } G \}$ is described and a combinatorial
basis $\B$ for $\RR E/ IG$ is identified. Using these, the $k$-th
theta body, $\textup{TH}_k(IG)$, of $IG$ can be described as:
 $$\left\{ \yy \in \RR^E \,:\,
 \begin{array}{l} 
   \exists \, M \succeq 0, \, M \in \RR^{|\B_k| \times
     |\B_k|}\,\,\textup{such that} \\  
   M_{\emptyset,\emptyset} = 1\\
   M_{F_1,F_2} = M_{F_3,F_4} \,\textup{ if } F_1 \Delta F_2 \Delta F_3
   \Delta F_4 \,\textup{ is a cycle in } G
 \end{array}
 \right \}.$$
 These theta bodies provide a new canonical set of SDP relaxations for
 $\textup{CUT}(G)$ that exploits the structure of $G$ directly. It is
 also shown in \cite{GLPT} that $IG$ is $\textup{TH}_1$-exact if and
 only if $G$ has no $K_5$-minor and no induced cycle of length at
 least five which answers Problem~8.4 posed by Lov{\'a}sz in
 \cite{Lovasz}.

\begin{remark} 
  We remark that the stable set problem and the first formulation of
  the max cut problem discussed above are special cases of the
  following general setup.  Let $\Delta$ be an {\em abstract
    simplicial complex} (or {\em independence system}) with vertex set
  $[n]$ recorded as a collection of subsets of $[n]$, called the {\em
    faces} of $\Delta$.  The {\em Stanley-Reisner} ideal of $\Delta$
  is the ideal $J_\Delta$ generated by the squarefree monomials
  $x_{i_1}x_{i_2} \cdots x_{i_k}$ such that $\{i_1, i_2, \ldots, i_k\}
  \subseteq [n]$ is not a face of $\Delta$.  If $I_\Delta := J_\Delta
  + \langle x_i^2 - x_i \,:\, i \in [n] \rangle$, then
  $\V_{\RR}(I_\Delta) = \{ \ss \in \{0,1\}^n \,:\,
  \textup{support}(\ss) \in \Delta \}$.  For $T \subseteq [n]$, recall
  that $\xx^T := \prod_{i \in T} x_i$. Then $\B := \{ \xx^T \,:\, T
  \in \Delta \} + I_{\Delta}$ is a basis for $\RR[\xx]/I_\Delta$
  containing $1+I_\Delta, x_1+I_\Delta, \ldots, x_n+I_\Delta$.
  Therefore, by Corollary~\ref{cor:simple version} and
  Proposition~\ref{prop:theta=Q}, the $k$-th theta body of $I_\Delta$
  is
$$\textup{TH}_k(I_\Delta) = \textup{proj}_{y_1, \ldots, y_n} \{ \yy
\in \RR^{\B_{2k}} \,:\, M_{\B_k}(\yy) \succeq 0, \, y_0 = 1\}.$$
Since
$\B$ is in bijection with the faces of $\Delta$, and $x_i^2 - x_i \in
I_\Delta$ for all $i \in [n]$, the theta body can be written
explicitly as follows:
$$\textup{TH}_k(I_\Delta) = \left\{ \yy \in \RR^n \,:\,
\begin{array}{l} 
\exists \, M \succeq 0, \, M \in \RR^{|\B_k| \times |\B_k|}
  \,\textup{such that} \\ 
M_{\emptyset\emptyset} = 1,\\
M_{\emptyset \{i\}} = M_{\{i\} \emptyset} = M_{\{i\}\{i\}} = y_i \\
M_{UU'} = 0 \,\,\textup{if} \,\,U \cup U' \not \in \Delta\\
M_{UU'} = M_{W W'} \,\,\textup{if} \,\, U \cup U' = W \cup W' 
\end{array}
\right \}.$$
If the dimension of $\Delta$ is $d-1$ (i.e., the largest faces in
$\Delta$ have size $d$), then $I_\Delta$ is $(1,d)$-sos and therefore,
$\textup{TH}_d$-exact since all elements of $\B$ have degree at most
$d$.  However, the theta-rank of $I_\Delta$ could be much less than
$d$.
\end{remark}


\section{Vanishing ideals of finite 
 sets of points} \label{sec:structure}

Recall that when $S \subset \RR^n$ is finite, its vanishing ideal
$\I(S)$ is zero-dimensional and real radical.

\begin{definition} \label{def:exact} We say that a finite set 
  $S \subset \RR^n$ is {\it exact} if its vanishing ideal $\I(S)
  \subseteq \RR[\xx]$ is $\textup{TH}_1$-exact.
\end{definition}

We now answer Lov{\'a}sz's question (Problem~\ref{prob:lovasz}) for
vanishing ideals of finite point sets in $\RR^n$.

\begin{theorem} \label{thm:perfect}
  For a finite set $S \subset \RR^n$, the following are equivalent.
\begin{enumerate} 
\item $S$ is exact.
\item $\I(S)$ is $(1,1)$-sos.
\item There is a linear inequality description of $\textup{conv}(S)$,
  of the form $g_i(x) \geq 0 \ (i=1,\ldots,m),$ where each $g_i$ is
  $1$-sos mod $\I(S)$.
\item There is a linear inequality description of $\textup{conv}(S)$,
  of the form $g_i(x) \geq 0 \ (i=1,\ldots,m),$ where each $g_i$ is an
  idempotent mod $\I(S)$, i.e., $g_i^2-g_i \in \I(S)$ for
  $i=1,\ldots,m$.
\item There is a linear inequality description of $\textup{conv}(S)$,
  of the form $g_i(x) \geq 0 \ (i=1,\ldots,m),$ where each $g_i$ takes
  at most two different values in $S$, i.e., for each $i$, $S$ is
  contained in the union of the hyperplane $g_i(\xx)=0$ and one unique
  parallel translate of it.
\end{enumerate}
\end{theorem}

\begin{proof}
  Since $\I(S)$ is real radical, by Corollary~\ref{cor:real radical
    equivalence}, (1) $\Leftrightarrow$ (2).
  
  The implication (2) $\Rightarrow$ (3) follows from the fact that
  $\textup{conv}(S)$ has a finite linear inequality description, since
  $S$ is finite. The implication (3) $\Rightarrow$ (2) follows from
  Farkas lemma, which implies that any valid inequality on $S$ is a
  non-negative real combination of the linear inequalities $g_i(x)
  \geq 0$.
  
  Suppose (3) holds and $\textup{conv}(S)$ is a full-dimensional
  polytope.  Let $F$ be a facet of $\textup{conv}(S)$, and $g(\xx)\geq
  0$ its defining inequality in the given description of
  $\textup{conv}(S)$.  Then $g(\xx)$ is $1$-sos mod $\I(S)$ if and
  only if there are linear polynomials $h_1,\ldots,h_l \in \RR[\xx]$
  such that $g \equiv h_1^2 + \cdots + h_l^2 \mod \I(S)$. In
  particular, since $g(\xx)=0$ on the vertices of $F$, and all the
  $h_i^2$ are non-negative, each $h_i$ must be zero on all the
  vertices of $F$.  Hence, since the $h_i$'s are linear, they must
  vanish on the affine span of $F$ which is the hyperplane defined by
  $g(\xx)=0$.  Thus each $h_i$ must be a multiple of $g$ and $g \equiv
  \alpha g^2$ mod $\I(S)$ for some $\alpha > 0$. We may assume that
  $\alpha = 1$ by replacing $g(\xx)$ by $g'(\xx) := \alpha g(\xx)$.
  If $\textup{conv}(S)$ is not full-dimensional, then since mod
  $\I(S)$, all linear polynomials can be assumed to define hyperplanes
  whose normal vectors are parallel to the affine span of $S$, the
  proof still holds. Therefore, (3) implies (4). Conversely, since if
  for a linear polynomial $g$, $g \equiv g^2$ mod $\I(S)$, then $g$ is
  $1$-sos mod $\I(S)$, (4) implies (3).
  
  The equivalence (4) $\Leftrightarrow$ (5) follows since $g \equiv
  g^2$ mod $\I(S)$ if and only if $g(\ss)(1 - g(\ss)) = 0 \,\,\forall
  \,\, \ss \in S$.
\end{proof}

Recall from the discussion at the end of Section 2.1 that by results
of Parrilo, if $I \subseteq \RR[\xx]$ is a zero-dimensional radical
ideal, then the theta-rank of $I$ is at most $|\V_{\CC}(I)|-1$. Better
upper bounds can be derived using the following extension of Parrilo's
theorem.

\begin{remark} \label{rem:k parallel translates} Suppose $S \subseteq
  \RR^n$ is a finite set such that each facet $F$ of
  $\textup{conv}(S)$ has a facet defining inequality $h_F(\xx) \geq 0$
  where $h_F$ takes at most $t+1$ values on $S$, then $\I(S)$ is
  $\textup{TH}_t$-exact: In this case, it is easy to construct a
  degree $t$ intepolator $g$ for the values of $\sqrt{h_F}$ on $S$,
  and we have $h_F \equiv g^2$ mod $\I(S)$. The result then follows
  from Farkas Lemma.
\end{remark}


\begin{remark} \label{rem:oddcycle} The theta-rank of $\I(S)$ could be
  much smaller than the upper bound in Remark~\ref{rem:k parallel
    translates}. Consider a $(2t+1)$-cycle $G$ and the set $S_G$ of
  characteristic vectors of its stable
  sets. Proposition~\ref{prop:2sos} shows that $\I(S_G)$ is
  $\textup{TH}_2$-exact. However, we need $t+1$ translates of the
  facet cut out by $\sum_{i=1}^{2t+1} x_i = t$ to cover $S_G$.
\end{remark}

In the rest of this section we derive various consequences of
Theorem~\ref{thm:perfect}. Finite point sets with property (5) in
Theorem~\ref{thm:perfect} have been studied in various contexts. In
particular, Corollaries \ref{cor:perfect}, \ref{cor:perfectgraph} and
\ref{cor:down-closed perfect} below were also observed independently
by Greg Kuperberg, Raman Sanyal, Axel Werner and G{\"u}nter Ziegler
(personal communication). In their work, $\textup{conv}(S)$ is called
a {\bf 2-level polytope} when property (5) in
Theorem~\ref{thm:perfect} holds.

If $S$ is a finite subset of $\ZZ^n$ and $\mathcal L$ is the smallest
lattice in $\ZZ^n$ containing $S$, then the lattice polytope
$\textup{conv}(S)$ is said to be {\bf compressed} if every {\em
  reverse lexicographic} triangulation of the lattice points in
$\textup{conv}(S)$ is {\em unimodular} with respect to $\mathcal L$.
Compressed polytopes were introduced by Stanley
\cite{StanleyCompressed}.  Corollary~\ref{cor:perfect} (4) and Theorem
2.4 in \cite{Sullivant} (see also the references after Theorem 2.4 in
\cite{Sullivant} for earlier citations of part or unpublished versions
of this result), imply that a finite set $S \subset \RR^n$ is exact if
and only if $\textup{conv}(S)$ is affinely equivalent to a compressed
polytope.
  
\begin{corollary} \label{cor:perfect} Let $S,S' \subset \RR^n$ be
  exact sets. Then
\begin{enumerate}
\item all points of $S$ are vertices of $\textup{conv}(S)$,
\item the set of vertices of any face of $\textup{conv}(S)$ is again
  exact, 
\item the product $S \times S'$ is exact, and
\item $\textup{conv}(S)$ is affinely equivalent to a $0/1$ polytope.
\end{enumerate}
\end{corollary}

\begin{proof} 
  The first three properties follow from Theorem~\ref{thm:perfect}
  (5).  If the dimension of $\textup{conv}(S)$ is $d \,\,(\leq n)$,
  then $\textup{conv}(S)$ has at least $d$ non-parallel facets. If
  $\aa \cdot \xx \geq b$ cuts out a facet in this collection, then
  $\textup{conv}(S)$ is supported by both $\{ \xx \in \RR^n \,:\, \aa
  \cdot \xx = b\}$ and a parallel translate of it. Taking these two
  parallel hyperplanes from each of the $d$ facets gives a
  parallelepiped. By Theorem~\ref{thm:perfect}, $S$ is contained in
  the vertices of this parallelepiped intersected with the affine hull
  of $S$. This proves (4).
\end{proof}

By Corollary~\ref{cor:perfect} (4), it essentially suffices to look at
subsets of $\{0,1\}^n$ to obtain all exact finite varieties in
$\RR^n$.  In $\RR^2$, the set of vertices of any $0/1$-polytope verify
this property.  In $\RR^3$ there are eight full-dimensional
$0/1$-polytopes up to affine equivalence. In Figures~\ref{fig:perfect
  in r3} and \ref{fig:non-perfect in r3} the convex hulls of the exact
and non-exact $0/1$ configurations in $\RR^3$ are shown.

\begin{figure}
\includegraphics[scale=0.35]{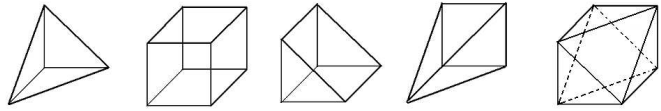}
\caption{Convex hulls of exact $0/1$ point sets in $\RR^3$.}
\label{fig:perfect in r3}
\end{figure}
       
\begin{figure}
\includegraphics[scale=0.35]{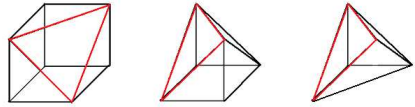}
\caption{Convex hulls of non-exact $0/1$ point sets in $\RR^3$.}
\label{fig:non-perfect in r3} 
\end{figure}

\begin{example} \label{ex:perfect}
  The vertices of the following $0/1$-polytopes in $\RR^n$ are exact
  for every $n$: (1) hypercubes, (2) (regular) cross polytopes, (3)
  hypersimplices (includes simplices), (4) joins of $2$-level
  polytopes, and (5) stable set polytopes of perfect graphs on $n$
  vertices.
\end{example}



\begin{theorem} \label{thm:2^n facets}
  If $S$ is a finite exact point set then $\textup{conv}(S)$ has at
  most $2^d$ facets and vertices, where $d=\dim \textup{conv}(S)$.
  Both bounds are sharp.
\end{theorem}

\begin{proof}
  The bound on the number of vertices is immediate by
  Corollary~\ref{cor:perfect} (4) and is achieved by $[0,1]^d$.
  
  For a polytope $P$ with an exact vertex set $S$, define a {\bf face
    pair} to be an unordered pair $(F_1,F_2)$ of proper faces of $P$
  such that $S \subseteq F_1 \cup F_2$ and $F_1$ and $F_2$ lie in
  parallel hyperplanes, or equivalently, there exists a linear form
  $h_{F_1,F_2}(\xx)$ such that $h_{F_1,F_2}(F_1)=0$ and
  $h_{F_1,F_2}(F_2)=1$. We will show that if $\textup{dim} \,\,P = d$
  then $P$ has at most $2^d-1$ face pairs and $2^d$ facets.
  
  If $d=1$, then an exact $S$ consists of two distinct points and $P$
  has two facets and one face pair as desired. Assume the result holds
  for $(d-1)$-polytopes with exact vertex sets and consider a
  $d$-polytope $P$ with exact vertex set $S$. Let $F$ be a facet of
  $P$ which by Theorem \ref{thm:perfect}, is in a face pair $(F,F')$
  of $P$.  Since exactness does not depend on the affine embedding, we
  may assume that $P$ is full-dimensional and that $F$ spans the
  hyperplane $\{\xx \,:\, x_{d}=0 \}$, while $F'$ lies in $\{\xx \,:\,
  x_{d}=1\}$. By Corollary \ref{cor:perfect}, $F$ satisfies the
  induction hypothesis and so has at most $(2^{d-1}-1)$ face pairs.
  Any face pair of $P$ besides $(F,F')$ induces a face pair of $F$ by
  intersection with $F$, and every facet of $P$ is in a face pair of
  $P$ since $S$ is exact.  The plan is to count how many face pairs of
  $P$ induce the same face pair of $F$ and the number of facets they
  contain.
  
  Fix a face pair $(F_1,F_2)$ of $F$, with associated linear form
  $h_{F_1,F_2}$ depending only on $x_1,\ldots,x_{d-1}$. Suppose
  $(F_1,F_2)$ is induced by a face pair of $P$ with associated linear
  form $H(\xx)$. Since $H$ and $h_{F_1,F_2}$ agree on every vertex of
  $F$, a facet of $P$, $H(\xx)=h_{F_1,F_2}(x_1,\ldots,x_{d-1})+cx_{d}$
  for some constant $c$.
  
  If $h_{F_1,F_2}(x_1,\ldots,x_{d-1})$ takes the same value $v$ on all
  of $F'$, then $H(F')=v+c = 0 \,\textup{or} \,1$ which implies that
  $c=-v$ or $c=1-v$. The two possibilities lead to the face pairs
  $(\textup{conv}(F_1 \cup F'), F_2)$ and $(\textup{conv}(F_2 \cup
  F'), F_1)$ of $P$.  Each such pair contains at most one facet of
  $P$.
  
  If $h_{F_1,F_2}(x_1,\ldots,x_{d-1})$ takes more than one value on the
  vertices of $F'$, then these values must be $v$ and $v+1$ for some
  $v$ since $H$ takes values $0$ and $1$ on the vertices of $F'$.  In
  that case, $c=-v$, so $H$ is unique and we get at most one face pair
  of $P$ inducing $(F_1,F_2)$. This pair will contain at most two
  facets of $P$.

  Since there are at most $2^{d-1}-1$ face pairs in $F$, they give us
  at most $2(2^{d-1}-1)$ face pairs and facets of $P$. Since we have
  not counted $(F,F')$ as a face pair of $P$, and $F$ and $F'$ as
  possible facets of $P$, we get the desired result.  The bound on the
  number of facets is attained by cross-polytopes.
\end{proof}

\begin{remark}
  G{\"u}nter Ziegler has pointed out that our proof of
  Theorem~\ref{thm:2^n facets} can be refined to yield that $P$ (as
  used above) has $2^d-1$ face pairs if and only if it is a simplex
  and $2^d$ facets if and only if it is a regular cross-polytope.
\end{remark}

Recall that Problem~\ref{prob:lovasz} was inspired by perfect graphs.
Theorem~\ref{thm:perfect} adds to the characterizations of a perfect
graph (c.f. Theorem~\ref{thm:knownforperfectgraphs}) as follows.

\begin{corollary} \label{cor:perfectgraph}
  For a graph $G$, let $S_G$ denote the set of characteristic vectors
  of stable sets in $G$. Then the following are equivalent.
  \begin{enumerate}
  \item The graph $G$ is perfect.
  \item The stable set polytope, $\textup{STAB}(G)$, is a $2$-level
    polytope.
  \end{enumerate}
\end{corollary}

A polytope $P$ in $\RR^n_{\geq 0}$ is said to be {\bf down-closed} if
 for all $\vv \in P$ and  $\vv' \in \RR^n_{\geq 0}$ 
such that $v'_i \leq v_i$ for $i=1,\ldots,n$, $\vv' \in P$. For a
graph $G$, $\textup{STAB}(G)$ is a down-closed $0/1$-polytope, and $G$
is perfect if and only if the vertex set of $\textup{STAB}(G)$ is
exact. We now prove that all down-closed $0/1$-polytopes with exact
vertex sets are stable set polytopes of perfect graphs.

\begin{theorem} \label{thm:down-closed perfect}
  Let $P \subseteq \RR^n$ be a down-closed $0/1$-polytope and $S$ be
  its set of vertices. Then $S$ is exact if and only if all facets of
  $P$ are either defined by non-negativity constraints on the
  variables or by an inequality of the form $\sum_{i \in I} x_i \leq
  1$ for some $I \subseteq [n]$.
\end{theorem}

\begin{proof} If $P$ is not full-dimensional then since it is
  down-closed, it must be contained in a coordinate hyperplane $x_i =
  0$ and the arguments below can be repeated in this lower-dimensional
  space. So we may assume that $P$ is $n$-dimensional. Then since $P$
  is down-closed, $S$ contains $\{{\bf 0},\ee_1, \ldots, \ee_n\}$.
  
  If all facets of $P$ are of the stated form, using that $S \subseteq
  \{0,1\}^n$, it is straight forward to check that $S$ is exact.
  
  Now assume that $S$ is exact and $g(\xx) \geq 0$ is a facet
  inequality of $P$ that is not a non-negativity constraint. Then
  $g(\xx) := c - \sum_{i=1}^{n} a_i x_i \geq 0$ for some integers $c,
  a_1, \ldots, a_n$ with $c \neq 0$. Since ${\bf 0} \in S$ and $S$ is
  exact, we get that $g(\ss)$ equals $0$ or $c$ for all $\ss \in S$.
  Therefore, for all $i$, $g(\ee_i) = c-a_i$ equals $0$ or $c$, so 
  $a_i$ is either $0$ or $c$. Dividing through by $c$, we get that the
  facet inequality $g(\xx) \geq 0$ is of the form $\sum_{i \in I} x_i
  \leq 1$ for some $I \subseteq [n]$.
\end{proof}

\begin{corollary} \label{cor:down-closed perfect}
  Let $P \subseteq  \RR^n$ be a full-dimensional down-closed
  $0/1$-polytope and $S$ be its vertex set. Then $S$ is exact if and
  only if $P$ is the stable set polytope of a perfect graph.
\end{corollary}

\begin{proof}
  By Corollary~\ref{cor:perfectgraph} we only need to prove the
  ``only-if'' direction. Suppose $S$ is exact. Then by
  Theorem~\ref{thm:down-closed perfect}, all facet inequalities of $P$
  are either of the form $x_i \geq 0$ for some $i \in [n]$ or $\sum_{i
    \in I} x_i \leq 1$ for some $I \subseteq [n]$. Define the graph
  $G=([n],E)$ where $\{i,j\} \in E$ if and only if $\{i,j\} \subseteq
  I$ for some $I$ that indexes a facet inequality of $P$.
  
  We prove that $P = \textup{STAB}(G)$ and that $G$ is perfect. Let $K
  \subseteq [n]$ such that its characteristic vector $\chi^K \in S$.
  If there exists $i,j \in K$ such that $i,j \in I$ for some $I$ that
  indexes a facet inequality of $P$, then $1-\sum_{i \in I} x_i$ takes
  three different values when evaluated at the points ${\bf 0}, \ee_i,
  \chi^K$ in $S$ which contradicts that $S$ is exact. 
  Therefore, $K$
  is a stable set of $G$ and $P \subseteq \textup{STAB}(G)$. If $K
  \subseteq [n]$ is a stable set of $G$ then, by construction, for
  every $I$ indexing a facet inequality of $P$, $\chi^K$ lies on
  either $\sum_{i \in I} x_i = 1$ or $\sum_{i \in I} x_i = 0$.
  Therefore $\chi^K \in P$ and $\textup{STAB}(G) \subseteq P$. Since
  all facet inequalities of $\textup{STAB}(G)$ are either
  non-negativities or clique inequalities, $G$ is perfect by
  \cite[Theorem~9.2.4 iii.]{GLS}.
\end{proof}


\section{Arbitrary $\textup{TH}_1$-exact Ideals} \label{sec:structure  
  arbitrary S} 

In this last section we describe $\textup{TH}_1(I)$ for an arbitrary
(not necessarily real radical or zero-dimensional) ideal $I \subseteq
\RR[\xx]$. The main structural result is
Theorem~\ref{thm:intersection} which allows the construction of
non-trivial high-dimensional $\textup{TH}_1$-exact ideals as in
Example~\ref{ex:infinite perfect set}.

In this study, the {\em convex quadrics} in $\RR[\xx]$ play a
particularly important role. These are precisely the polynomials of
degree two that can be written as $F(\xx)=\xx^t A \xx + \bb^t \xx +
c$, where $A \not = 0$ is an $n \times n$ positive semidefinite matrix, $\bb
\in \RR^n$ and $c \in \RR$.  Note that every sum of squares of linear
polynomials in $\RR[\xx]$ is a convex quadric.

\begin{lemma}\label{lem:contained}
  For $I \subseteq \RR[\xx]$, $\textup{TH}_1(I) \neq \RR^n$ if and
  only if there exists some convex quadric $F \in I$.
\end{lemma}

\begin{proof}
  If $\textup{TH}_1(I) \neq \RR^n$, there exists a degree one polynomial $f$
  that is strictly positive on $\textup{TH}_1(I)$, hence $1$-sos modulo $I$.
  Then $f(\xx) \equiv g(\xx)$ mod $I$ for some $1$-sos $g(\xx) \not = 0$ and 
  $g(\xx)-f(\xx) \in I$ is a convex quadric.
  
  Conversely, suppose $\xx^tA\xx + \bb^t \xx + c \in I$ with $A
  \succeq 0$. Then for any $\dd \in \RR^n$,
  $$(\xx+\dd)^tA(\xx+\dd) = \xx^tA\xx + 2\dd^tA\xx +\dd^tA\dd \equiv
  (2\dd^tA-\bb^t) \xx + \dd^tA\dd - c \mod I.$$
  Therefore, since
  $(\xx+\dd)^tA(\xx+\dd)$ is a sum of squares of linear polynomials,
  the linear polynomial $(2\dd^tA-\bb^t) \xx + \dd^tA\dd - c $ is
  $1$-sos mod $I$ and $\textup{TH}_1(I)$ must satisfy it. Since $\dd$
  can be chosen so that $(2\dd^tA-\bb^t) \neq 0$, $\textup{TH}_1(I)$
  is not trivial.
\end{proof}




\begin{lemma}\label{lem:theta_intersection}
  For an ideal $I \subseteq \RR[\xx]$, $\textup{TH}_1(I)=\bigcap
  \textup{TH}_1(\langle F \rangle)$, where $F$ varies over all convex
  quadrics in $I$.
\end{lemma}
\begin{proof}
  If $F \in I$ then $\langle F \rangle \subseteq I$. Also, if $f$
  is linear and $1$-sos mod $\langle F \rangle$ then it is also
  $1$-sos mod $I$. Therefore, $\textup{TH}_1(I) \subseteq
  \textup{TH}_1(\langle F \rangle)$.
  
  To prove the reverse inclusion, we need to show that if $f$ is a linear
  polynomial that is nonnegative on $\textup{TH}_1(I)$, it is also nonnegative 
  on $\bigcap_{F \in I} \textup{TH}_1(\langle F \rangle)$, where $F$
  is a convex quadric. It suffices to show that whenever $f$ is linear
  and $1$-sos mod $I$, then there is a convex quadric $F \in I$ such
  that $f(\xx) \geq 0$ is valid for $\textup{TH}_1(\langle F
  \rangle)$, or equivalently that $f$ is $1$-sos mod $\langle F
  \rangle$. Since $f$ is $1$-sos mod $I$, there is a sum of squares of
  linear polynomials $g(\xx)$ such that $f(\xx) \equiv g(\xx)$ mod
  $I$. But $g$ is a convex quadric, hence so is $g(\xx)-f(\xx)$.  Thus
  $f$ is $1$-sos mod the ideal $\langle g(\xx) - f(\xx) \rangle$ and
  we can take $F(\xx) = g(\xx)-f(\xx)$.
\end{proof}

\begin{lemma} \label{lemma:theta_convex}
  If $F(\xx)=\xx^tA\xx + \bb^t\xx + c$ with $A \succeq 0$, then
  $\textup{TH}_1(\langle F \rangle)=\textup{conv}(\V_{\RR}(F))$.
\end{lemma}
\begin{proof}
  We know that $\textup{conv}(\V_{\RR}(F)) \subseteq
  \textup{TH}_1(\langle F \rangle)$ and, since $F$ is convex,
  $\textup{conv}(\V_{\RR}(F))=\{\xx \in \RR^n \,:\, F(\xx)\leq 0\}$.
  Thus, if for every $\xx \in \V_{\RR}(F)$ $\textup{grad}F(\xx) \neq
  {\bf 0}$, then $\textup{conv}(\V_{\RR}(F))$ is supported by the
  tangent hyperplanes to $\V_{\RR}(F)$. In this case, to show that
  $\textup{TH}_1(\langle F \rangle) \subseteq
  \textup{conv}(\V_{\RR}(F))$, it suffices to prove that the defining
  (linear) polynomials of all tangent hyperplanes to $\V_{\RR}(F)$ are
  $1$-sos mod $\langle F \rangle$.  The proof of the ``if'' direction
  of Lemma~\ref{lem:contained} shows that it would suffice to prove
  that a tangent hyperplane to $\V_{\RR}(F)$ has the form
  $(2\dd^tA-\bb^t) \xx + \dd^tA\dd - c = 0$, for some $\dd \in
  \RR^n$. The tangent at $\xx_0 \in \V_{\RR}(F)$ has equation
  $0=(2A\xx_0+\bb)^t(\xx-\xx_0)$ which can be rewritten as
  $$0= (2\xx_0^tA+\bb^t)\xx - 2\xx_0^tA\xx_0 -\bb^t\xx_0 =
  (2\xx_0^tA+\bb^t)\xx -\xx_0^tA\xx_0 + c,$$
  and so setting
  $\dd=-\xx_0$ gives the result. 
  
  Suppose there is an $\xx_0$ such that $F(\xx_0) = 0$ and
  $\textup{grad} F(\xx_0) = {\bf 0}$. By translation we may assume
  that $\xx_0 = 0$, hence, $c = 0$ and $\bb = {\bf 0}$. Therefore $F =
  \xx^t A \xx = \sum h_i^2$ where the $h_i$ are linear.  Since
  $\V_{\RR}(\langle F\rangle)=\V_{\RR}(\langle h_1,\ldots,h_m\rangle)$ it
  is enough to prove that all inequalities $\pm h_i \geq 0$ are valid
  for $\textup{TH}_1(\langle F \rangle)$.  For any
  $\epsilon >0$ we have
  $$(\pm h_l + \epsilon)^2 + \sum_{i \not = l} h_i^2 = F \pm 2
  \epsilon h_l + \epsilon^2 \equiv 2 \epsilon(\pm h_l + \epsilon/2)
  \,\,\textup{mod} \,\, \langle F \rangle,$$
  so $\pm h_l + \epsilon/2$
  is $1$-sos mod $\langle F \rangle$ for all $l$ and all $\epsilon
  >0$. This implies that all the inequalities $\pm h_l + \epsilon/2
  \geq 0$ are valid for $\textup{TH}_1(\langle F \rangle)$, therefore
  so are the inequalities $\pm h_l \geq 0$.
\end{proof}

\begin{theorem} \label{thm:intersection} Let $I \subseteq \RR[\xx]$ be
  any ideal, then 
  $$\textup{TH}_1(I)=\bigcap_{{F \in I}\atop{F \textrm{ convex
        quadric}}} \textup{conv}(\V_{\RR}(F)) = \bigcap_{{F \in
      I}\atop{F \textrm{ convex quadric}}} \{ \xx \in \RR^n : F(\xx)
  \leq 0\}.$$
\end{theorem}
\begin{proof}
  Immediate from Lemma \ref{lem:theta_intersection} and Lemma
  \ref{lemma:theta_convex}.
\end{proof}

\begin{example} \label{ex:infinite perfect set}
  Theorem~\ref{thm:intersection} shows that some non-principal ideals
  such as $I=\langle x^2-z,y^2-z \rangle \subseteq \RR[x,y,z]$ are
  $\textup{TH}_1$-exact. Since $\V_{\RR}(I)=\{(\pm t,\pm t,t^2): t
  \in \RR\}$, fixing the third coordinate we get the four points
  $(x,y,t^2)$ where $|x|=|y|=|t|$ which implies that 
  $$\textup{conv}(\V_{\RR}(I))\supseteq \{(x,y,t^2): |x| \leq
  t, |y| \leq t, t \geq 0\}.$$ It is easy to see that the right hand
  side is equal to $\{(x,y,z): x^2 \leq z, y^2 \leq z\}$ which is
  exactly $\conv(\V_{\RR}(x^2-z)) \bigcap \conv(\V_{\RR}(y^2-z))$ and
  so contains $\textup{TH}_1(I)$ which contains
  $\textup{conv}(\V_{\RR}(I))$. So all inclusions must be equalities
  and $I$ is $\textup{TH}_1$-exact. This kind of reasoning allows us
  to construct non-trivial examples of $\textup{TH}_1$-exact ideals
  with high-dimensional varieties.
\end{example}


\begin{example} \label{ex:four points}
  Consider the set $S=\{(0,0),(1,0),(0,1),(2,2)\}$. Then the family of
  all quadratic curves in $\I(S)$ is
$$a(x^2-x) + b(y^2-y) - (\frac{a+b}{2}) xy = (x,y) \left[
\begin{array}{cc}
a & -(\frac{a+b}{4}) \\
 -(\frac{a+b}{4})  & b
\end{array}
\right] \left( \begin{array}{c} x \\ y \end{array} \right) -ax -by.$$
Since the case where both $a$ and $b$ are zero is trivial, we may
normalize by setting $a+b=1$ and get the matrix in the quadratic to be
$$\left[
\begin{array}{cc}
\lambda & -1/4 \\
 -1/4  & 1-\lambda
\end{array}
\right] $$
with $\lambda \geq 0$. This matrix is positive semidefinite
if and only if $\lambda(1-\lambda) - 1/16 \geq 0$, or equivalently, if
and only if $\lambda \in [1/2 - \sqrt{3}/4, 1/2 + \sqrt{3}/4]$.

This means that $(x,y) \in \textup{TH}_1(\I(S))$ if and only if, for
all such $\lambda$,
$$\lambda(x^2-x) + (1-\lambda)(y^2-y) -\frac{1}{2} xy \leq 0.$$
Since the
right-hand-side does not depend on $\lambda$, and the left-hand-side
is a convex combination of $x^2-x$ and $y^2-y$, the inequality holds
for every $\lambda \in [1/2 - \sqrt{3}/4, 1/2 + \sqrt{3}/4]$ if and
only if it holds at the end points of the interval. Equivalently, if
and only if
$$\left(\frac{1}{2} - \frac{\sqrt{3}}{4}\right)(x^2-x) +
\left(\frac{1}{2} + \frac{\sqrt{3}}{4}\right)(y^2-y) -\frac{1}{2}
xy \leq 0,$$
and
$$\left(\frac{1}{2} + \frac{\sqrt{3}}{4}\right)(x^2-x) +
\left(\frac{1}{2} - \frac{\sqrt{3}}{4}\right)(y^2-y) -\frac{1}{2}
xy \leq 0.$$
But this is just the intersection of the convex hull of the two
curves obtained by turning the inequalities into equalities.
Figure~\ref{fig:thetabody} shows this intersection.

\begin{figure} 
\includegraphics[scale=0.35]{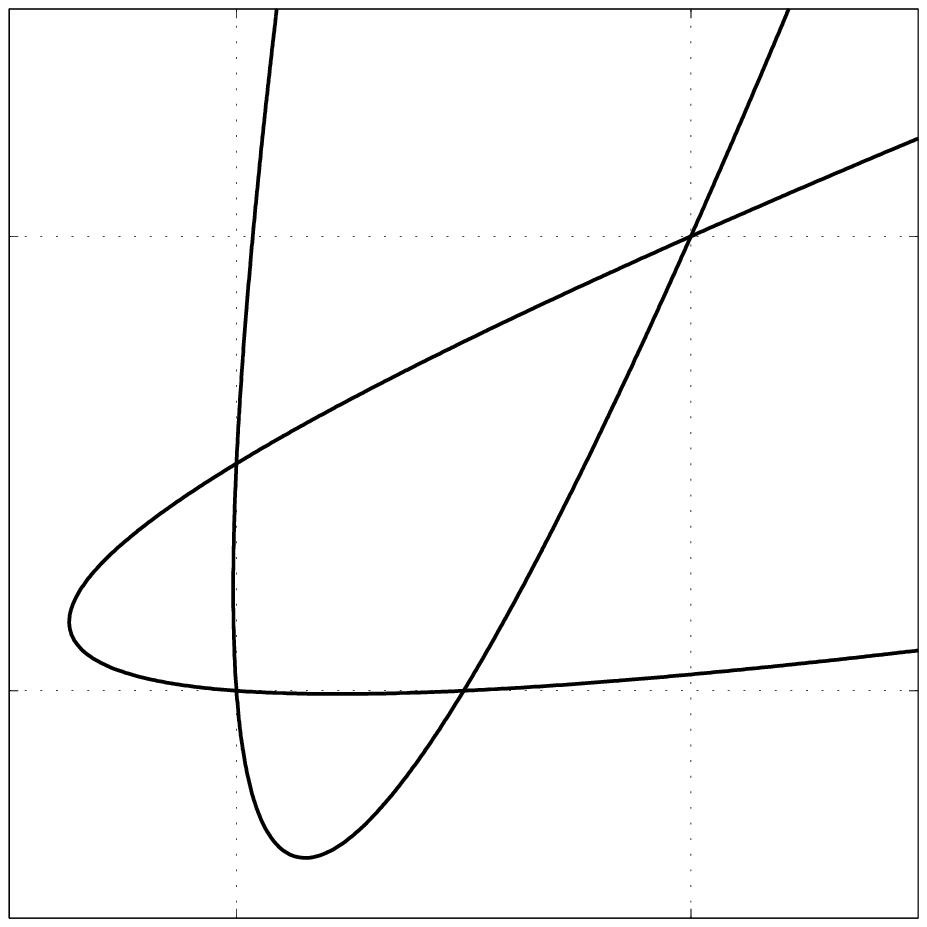}
\caption{Example~\ref{ex:four points}}
\label{fig:thetabody}
\end{figure}
\end{example}


\bibliographystyle{plain}

\end{document}